\documentclass[12pt,a4paper,twoside]{amsart}
\usepackage{geometry}
\usepackage{tabls}

\usepackage{url}

\usepackage{amssymb}
\usepackage{mathrsfs}
\usepackage{amsmath}
\usepackage{amsthm}
\usepackage{amscd}
\usepackage{fancyhdr}
\usepackage{enumerate}
\usepackage{paralist}
\usepackage{graphicx}
\usepackage{cite}

\usepackage{footmisc}
\numberwithin{equation}{section}

\theoremstyle{plain}
\newtheorem{Thm}{Theorem}[section]
\newtheorem{lem}[Thm]{Lemma}
\newtheorem{prop}[Thm]{Proposition}

\newtheorem{Def}[Thm]{Definition}
\newtheorem{rem}[Thm]{Remark}

\usepackage{booktabs}
\usepackage{fancyvrb}

\newcommand{\Rea}{\operatorname{Re}}

\newcommand{\ones}{\{1\}}
\newcommand{\dd}{\,\mathrm d}

\begin{document}

\title{Divided differences and  complex variations of  multiple zeta-star values}

\author{Jiangtao Li}

\email{lijiangtao@csu.edu.cn}
\address{Jiangtao Li \\ School of Mathematics and Statistics, HNP-LAMA, Central South University, Hunan Province, China}

\begin{abstract}
Motivated by the known derived set $[1,+\infty)$ of multiple
zeta-star values, we study finite multiple star harmonic sums and their
infinite-depth limits. Divided differences give a cube-integral
representation of their holomorphic interpolation on
$\mathrm{Re}(s)>0$, which agrees with the Kawashima functions after
reversal of the index. The finite-depth interpolations converge locally uniformly in $s$
as the depth tends to infinity, uniformly over all positive integer
index sequences on compact subsets. For real $s>1$, these limits define a bijection
from the infinite positive integer indices onto $(1,s]$.
For each admissible infinite index, they tend to the corresponding
multiple zeta-star limit as $\mathrm{Re}(s)\to+\infty$ with bounded
imaginary part. For $\mathrm{Re}(s)>1$, we prove strict separation
of every ordered pair of finite indices and of every ordered pair of
infinite indices: either the difference or the difference divided by
$s$ has positive real part throughout this half-plane, with the
choice fixed by the pair. In particular, both index maps are
injective. The proof uses a summation principle for functions with
positive real part, established through positive integral
representations and explicit kernel estimates.
\end{abstract}

\let\thefootnote\relax\footnotetext{
2020 $\mathnormal{Mathematics} \;\mathnormal{Subject}\;\mathnormal{Classification}$. 11M32, 41A05.\\
$\mathnormal{Keywords}$:  Multiple zeta-star values, divided differences. \\
}

\maketitle

\section{Introduction}\label{introduction}
      Multiple zeta values are defined by
      \[
      \zeta(k_1,\cdots,k_r)=\sum_{n_1>\cdots>n_r>0}\frac{1}{n_1^{k_1}\cdots n_r^{k_r}},\qquad k_1\geq 2,k_2,\cdots, k_{r}\geq 1.
      \]
      For a multiple zeta value $\zeta(k_1,\cdots,k_r)$, denote by $N=k_1+\cdots+k_r$ and $r$ its weight  and depth respectively.

  For  $k_1\geq 2,k_2,\cdots, k_{r}\geq 1$, the multiple zeta-star value $\zeta^{\star}(k_1,\cdots,k_r)$ is defined by
        \[
        \zeta^{\star}(k_1,\cdots,k_r)=\sum_{n_1\geq\cdots\geq n_r\geq 1}\frac{1}{n_1^{k_1}\cdots n_r^{k_r}}.
        \]
        The theory of multiple zeta-star values has been developed in parallel with, but not merely as a byproduct of, the theory of multiple zeta values. Its algebraic foundations are closely related to Hoffman’s work on multiple harmonic series and harmonic algebras \cite{hof1},\cite{hof2}. Important structural and identity results were obtained by Ohno-Wakabayashi \cite{OW}, Ohno-Zudilin \cite{OZ}, Ihara-Kajikawa-Ohno-Okuda \cite{IKOO}, Muneta \cite{mun},  Kaneko-Ohno \cite{KO},  Yamamoto \cite{yam1}, Zhao \cite{zha}, and others. 
        
                In \cite{lit}, the author  discovered the following bijective map: \\
        Denote by $$\mathbb{Z}^{+}=\{1,2,\cdots, n,\cdots\}$$ and 
        \[
        \begin{split}
        \mathcal{T}=\big\{(k_1,k_2,\ldots)\in(\mathbb{Z}^+)^\infty:
        &\ k_1\geq2,\ \text{and if }k_1=2,\\
        &\ k_i\geq2\text{ for some }i\geq2\big\}.
        \end{split}
        \]
        There is a bijective map
        \[
        \eta: \mathcal{T}\rightarrow (1,+\infty),\]
  \[
  {\bf k}=(k_1,\cdots,k_r,\cdots)\mapsto \eta({\bf{k}})=\mathop{\mathrm{lim}}_{r\rightarrow +\infty} \zeta^{\star}(k_1,\cdots,k_r).
                \]
                We call the map $\eta$ zeta-star correspondence. In fact, one can compare the theory of zeta-star correspondence to the theory of continued fractions. The theory of zeta-star correspondence is related to fractal geometry, Diophantine approximation, multiple polylogarithms.  For more details, see\cite{kam}, \cite{lia}, \cite{lid},\cite{lir},\cite{lit}.  Hirose, Murahara and Onozuka \cite{hmo}  discovered the above bijective map independently.

 For $(k_1,\cdots,k_r)\in (\mathbb{Z}^+)^r$  and $n\in \mathbb{Z}^+$, define 
 \[
 H_n^\star(  \varnothing )=1,\;\;\;\;H_n^\star(k_1, \cdots,k_r)=\sum_{n\geq m_1\geq \cdots \geq m_r\geq 1} \frac{1}{m_1^{k_1}\cdots m_r^{k_r}}.
 \]
  For a fixed $n\geq 2$, is there any kind of zeta-star correspondence map for $H_n^\star$?
 \begin{Thm}\label{int}
 Denote by $\widehat{\mathcal{T}}=(\mathbb{Z}^+)^\infty$, for $n\in \mathbb{Z}^+$ and $n\geq 2$, there is a bijective map 
 \[\mathfrak{h}_n: \widehat{\mathcal{T}}\rightarrow (1, n],
 \]
 \[
  {\bf k}=(k_1,\cdots,k_r,\cdots)\mapsto \mathfrak{h}_n({\bf{k}})=\mathop{\mathrm{lim}}_{r\rightarrow +\infty} H_n^{\star}(k_1,\cdots,k_r).
                \]
  \end{Thm}
 
 In fact, for $n=2$, Theorem \ref{int} is essentially another reformulation of the binary expansion. For fixed $(k_1,\cdots,k_r)\in (\mathbb{Z}^+)^r$, the multiple  harmonic sum 
 \[
 H_n^{\star}(k_1,\cdots,k_r) \]
 can be viewed as a function of the  integer variable $n$.  More generally, is there any kind of zeta-star correspondence map for $H_s^\star$ when $s$ is a complex variable instead of a positive integer?
 
 To achieve the above goals, we have to first  define $H_s^\star(k_1,\cdots,k_r)$.  For $\mathrm{Re}\;s>0$ and $k_1\in \mathbb{Z}^+$, define 
 \[
 G_n^\star(s;k_1)=\sum_{m=1}^n\frac{1}{(m+s)^{k_1}}  . \]
 
 \[
 H_s^\star(\varnothing)=1,\quad   H_s^\star(k_1)=\mathop{\mathrm{lim}}_{M\rightarrow +\infty} \bigg{(}H_M^\star(k_1)-G_M^\star(s;k_1)     \bigg{)}.
 \]
 It is clear that $$H_s^\star(k_1)\big{|}_{s=n}=H_n^\star(k_1)$$for $n\in  \mathbb{Z}^+$. 
 
 For $r\geq 2$,$(k_1,\cdots,k_r)\in(\mathbb{Z}^+)^r$ and $\mathrm{Re}(s)>0$, by induction, define
 \[
  G_n^\star(s;k_1,\cdots,k_r)=\sum_{n\geq m_1\geq \cdots \geq m_r\geq 1} \frac{1}{(m_1+s)^{k_1}\cdots (m_r+s)^{k_r}}.
     \]
 \[
H_s^\star(k_1, \cdots,k_r) = \mathop{\mathrm{lim}}_{M\rightarrow +\infty}  \Big{(}   H_M^\star(k_1,\cdots,k_r)-   \sum_{j=1}^r  G_M^\star(s;k_1,\cdots, k_j)H_s^\star (k_{j+1},\cdots, k_r) \Big{)} . \]
  
    In fact, we have 
  \begin{Thm}\label{noint}
 $(i)$ For $r\geq 1$, $(k_1,\cdots, k_r)\in (\mathbb{Z}^+)^r$, $H_s^\star(k_1,\cdots,k_r)$ is well-defined and analytic on
  $\mathrm{Re}(s)>0$. Moreover, for $n\in \mathbb{Z}^+$, 
  \[
  H_s^\star(k_1,\cdots,k_r)\big{|}_{s=n}= H_n^\star(k_1, \cdots,k_r). \]
  $(ii)$ Furthermore, let 
  \[
X_j=x_1x_2\cdots x_{k_1+\cdots+k_j},\qquad 1\leq j\leq r.  \]  
  Then one has the integral representation 
  \[
  H_s^\star(k_1,\cdots,k_r)=\int_{(0,1)^{k_1+\cdots+k_r}} [1,X_1,\cdots, X_r]t^{s+r-1}dx_1\cdots dx_{k_1+\cdots+k_r}
  \] on $\mathrm{Re}(s)>0$.   Here $[1,X_1,\cdots, X_r]t^{s+r-1}$ is the divided difference of $t^{s+r-1}$ relative to $1,X_1,\cdots, X_r$.
      \end{Thm} 
      
A divided difference is a generalization of the ordinary difference quotient to several points. It is a fundamental tool in interpolation, approximation theory and numerical analysis. For a survey about divided differences, see \cite{boo1}. We will give a brief introduction in Section \ref{ac}. 
    
   For $s=n$, the formula in Theorem~\ref{noint}, $(ii)$ gives a
divided-difference reformulation of integral representations of finite
multiple harmonic star sums. Such integral representations are known;
see Yamamoto~\cite[Theorem 2.1 and Proposition 2.8]{yam}.
The identity relating divided differences of monomials to complete
homogeneous symmetric polynomials is also standard~\cite{boo1}.
Our emphasis is on the cube-integral formulation, its use in infinite-depth
limits, and the order and separation properties established below.
  
    \begin{rem}
     The inductive definition of $H_s^\star(k_1,\cdots,k_r)$ is inspired by the following observation:
     For $s\in \mathbb{Z}^+$, the set 
     \[
     \Big{\{} (n_1,\cdots,n_r)\in \left( \mathbb{Z}^+\right)^r  \,\Big{|}\,n_1\geq \cdots \geq n_r\geq 1  \Big{\}}
     \]
     is the {\bf disjoint union} of the  subsets:
      \[
     \Big{\{} (n_1,\cdots,n_r)\in \left( \mathbb{Z}^+\right)^r  \,\Big{|}\,s\geq n_1\geq \cdots \geq n_r\geq 1  \Big{\}},
     \]
      \[
     \Big{\{} (n_1,\cdots,n_r)\in \left( \mathbb{Z}^+\right)^r  \,\Big{|}\,n_1\geq s+1, s\geq n_2\geq \cdots \geq n_r\geq 1  \Big{\}},
     \]
     \[\vdots\qquad \vdots \qquad \vdots\]
     \[
     \Big{\{} (n_1,\cdots,n_r)\in \left( \mathbb{Z}^+\right)^r  \,\Big{|}\,n_1\geq\cdots\geq n_{r-1}\geq  s+1, s \geq n_r\geq 1  \Big{\}},
     \]
      \[
     \Big{\{} (n_1,\cdots,n_r)\in \left( \mathbb{Z}^+\right)^r  \,\Big{|}\,n_1\geq \cdots \geq n_r\geq s+1  \Big{\}}.
     \]
    \end{rem}

    For $\mathrm{Re}(s)>0$, does the limit
 \[
 \mathop{\mathrm{lim}}_{r \rightarrow +\infty}   H_s^\star(k_1, \cdots,k_r) 
   \]  
   exist for any $(k_1,\cdots, k_r,\cdots)\in (\mathbb{Z}^+)^\infty$?  
 
   When $s$ is a complex number instead of positive integer, we have 
    \begin{Thm}\label{com}
    $(i)$ For $\mathrm{Re}(s)>0$,  there is a well-defined map
    \[\mathfrak{h}_s: \widehat{\mathcal{T}}\rightarrow \mathbb{C},
 \]
  \[
  {\bf k}=(k_1,\cdots,k_r,\cdots)\mapsto \mathfrak{h}_s({\bf{k}})=\mathop{\mathrm{lim}}_{r\rightarrow +\infty} H_s^{\star}(k_1,\cdots,k_r).
                \]
    The convergence is locally uniform in $s$ and uniform over all
    ${\bf k}\in\widehat{\mathcal T}$ on each compact subset of
    $\{s:\mathrm{Re}(s)>0\}$. In particular, for each fixed infinite
    index, $s\mapsto\mathfrak h_s({\bf k})$ is holomorphic there.\\
  $(ii)$ For $s\in \mathbb{R}$ and $s>1$, the above map induces a bijective map which we still refer to as $\mathfrak{h}_s$:
      \[\mathfrak{h}_s: \widehat{\mathcal{T}}\rightarrow (1,s],
 \]
  \[
  {\bf k}=(k_1,\cdots,k_r,\cdots)\mapsto \mathfrak{h}_s({\bf{k}})=\mathop{\mathrm{lim}}_{r\rightarrow +\infty} H_s^{\star}(k_1,\cdots,k_r).
                \]
 $(iii)$ For every fixed ${\bf k}\in\mathcal T$ and every $T>0$,
    \[
    \lim_{\sigma\to+\infty}\ \sup_{|\tau|\leq T}
    \bigl|\mathfrak h_{\sigma+i\tau}({\bf k})-\eta({\bf k})\bigr|=0.
    \]
    Thus the convergence to $\eta$ is pointwise in the infinite index,
    uniformly for the imaginary part in each bounded interval.
       \end{Thm}
    
     We call $\mathfrak h_s$ the finite zeta-star correspondence.
A motivating arithmetic question is to determine the infinite sequence
$\eta^{-1}(x)\in\mathcal T$ for a rational or algebraic number $x>1$.
Theorem~\ref{com} relates this question to a family of complex analytic
interpolations, although it does not by itself give an algorithm for
determining that inverse sequence. The arithmetic of periods provides
broader context for these questions~\cite{KZ}.

At $s=2$ the correspondence is described explicitly by
\[
 \mathfrak h_2(k_1,k_2,\ldots)
 =1+\sum_{j=1}^{\infty}2^{-(k_1+\cdots+k_j)}.
\]
Indeed, in a weakly decreasing tuple with entries in $\{1,2\}$, the
entries equal to $2$ form an initial segment. Thus the finite-depth
version of the displayed formula follows directly from the definition.
The formula encodes the binary expansion with infinitely many digits
equal to $1$, including the all-ones tail convention at dyadic values.
The limit in Theorem~\ref{com}, $(iii)$ connects this finite
correspondence to the zeta-star correspondence as $\mathrm{Re}(s)$
tends to infinity.

       By Remark \ref{nze}, one has  \[
         H_s^\star (k_1,\cdots,k_r)\neq 0    \]
       for $\mathrm{Re}(s)>0$. 
     In general, we have 
       \begin{Thm}\label{thm:complex-injectivity}
       $(i)$ For $(k_1,\cdots,k_p)\neq (l_1,\cdots,l_q)$, one has  \[
         H_s^\star (k_1,\cdots,k_p)-H_s^\star(l_1,\cdots,l_q)\neq 0    \]
       for $\mathrm{Re}(s)>1$.    \\
       $(ii)$ For $\mathrm{Re}(s)>1$, the map  $\mathfrak{h}_s$ is injective.
     \end{Thm}
     
     We also call the injective map $\mathfrak h_s$ the finite  zeta-star
     correspondence for convenience. The stronger ordered separation is proved in
     Theorems~\ref{thm:halfplane-order} and
     \ref{thm:infinite-halfplane-order}.

     \begin{rem}
     The motivation of complex variations of multiple zeta-star values is inspired by Riemann's marvellous research \cite{rie} on the theory of the Riemann zeta function. Through the systematic application of complex analysis, Riemann provided profound insights on the distribution of primes. We hope that through the finite Zeta-star correspondence, we  can have a better understanding of the original Zeta-star correspondence.
     \end{rem}
     
     \begin{rem}
     Kumar \cite{kum} investigated the derived set of the set of multiple zeta values.  Li and Pan \cite{LP} obtained the derived sets of the set of some $q$-analogue multiple zeta values.
     \end{rem}
     
     \begin{rem}
For $\overrightarrow\mu=(\mu_1,\ldots,\mu_r)$, Kawashima's
function~\cite{kaw} can be written as the Newton series
\[
 F_{\overrightarrow\mu}(z)=\sum_{n=0}^{\infty}(-1)^n
 (\Delta S_{\overrightarrow\mu})(n)\binom zn,
 \qquad S_{\overrightarrow\mu}(n)=H_n^\star(\mu_r,\ldots,\mu_1),
\]
where $S_{\overrightarrow\mu}(0)=0$ for a nonempty index and
\[
 (\Delta S_{\overrightarrow\mu})(n)
 =\sum_{j=0}^n(-1)^j\binom njS_{\overrightarrow\mu}(j).
\]
It satisfies
$F_{\overrightarrow\mu}(N)=H_N^\star(\mu_r,\ldots,\mu_1)$
at nonnegative integers. In fact,
\[
 F_{\overrightarrow\mu}(z)=H_z^\star(\mu_r,\ldots,\mu_1),
 \qquad \mathrm{Re}(z)>0.
\]
To justify this identification, use the absolutely convergent recursive
series~\eqref{eq:nested-series} proved below. After reversing the index,
it is exactly the recursive fraction series in
Yamamoto~\cite[Proposition 2.11]{yam}. Both families have the same
depth-one series and the same empty-index value $1$, so induction on
depth proves equality. Agreement at the positive integers alone would
not establish equality of arbitrary analytic interpolations.
The divided-difference representation therefore supplies another
description of this established family of functions. Kawashima's
analysis of products and Taylor coefficients at $z=0$ yields relations
among multiple zeta values~\cite{kaw,yam}.
     \end{rem}
     
     \section{Finite zeta-star correspondence in integer cases} \label{fzs}
     In this section, we will give a total order structure on the set of finite multiple star harmonic sums.   Theorem \ref{int} is proved as an application. The analysis in this section is similar to the analysis of multiple zeta-star values in \cite{lit}. The main references are \cite{lit,lid}.
     
     Denote by $$\widehat{\mathcal{S}}=\Big{\{}(k_1,\cdots,k_r)\,\Big{|}\,k_i\in\mathbb{Z}^{+}\ (1\leq i\leq r), r\geq 1\Big{\}}.$$
     Define the total order $\succ$ on $\widehat{\mathcal{S}}$ as follows.
     Every proper extension is larger than its prefix:
     \[
     (k_1,\cdots,k_r,k_{r+1},\cdots,k_{r+p})
        \succ(k_1,\cdots,k_r),\qquad p\geq1.
     \]
     If neither index is a prefix of the other, let $i$ be their first
     differing position. Then
     \[
     (k_1,\cdots,k_p)\succ(l_1,\cdots,l_q)
        \quad\Leftrightarrow\quad k_i<l_i.
     \]
     Thus the order is the reversed lexicographic order, with a proper
     extension larger than its prefix.
      If \[
     (k_1,\cdots,k_p)\succ (l_1,\cdots,l_q), 
     \]
     we also write
     as  
\[(l_1,\cdots,l_q)\; \rotatebox[origin=c]{180}{\text{ $\succ$}} (k_1,\cdots,k_p).\]     
     \begin{lem}\label{many1}
     For $n\geq 2$ and $r\geq 1$, one has 
     \[
     \begin{split}
     & H_n^\star(\{1\}^r)     =\sum_{n\geq m_1\geq \cdots \geq m_r\geq 1}  \frac{1}{m_1\cdots m_r}          < n.         \\
     \end{split}
     \]
     \end{lem}
       \noindent{\bf Proof:} 
     For $r=1$, $H_n^\star(1)<n$ is immediate. For $r\geq2$,
     \[
     \begin{split}
     &\;\;\;\;H_n^\star(\{1\}^r)\\
     &<\sum_{n\geq m_1\geq \cdots \geq m_r\geq 1}  \frac{1}{m_1\cdots m_{r-1}}=\sum_{n\geq m_1\geq \cdots\geq m_{r-1}\geq 1}\frac{m_{r-1}}{ m_1\cdots m_{r-1}    }=\cdots=n. \\
         \end{split}
         \]
            
        $\hfill\Box$\\    
        \begin{lem}\label{inf1}
        $(i)$ For $n\geq 1$, 
        \[
        \mathop{\mathrm{lim}}_{r\rightarrow +\infty} H_n^\star(\{1\}^r) =n.        \]
        $(ii)$ \[
        \mathop{\mathrm{lim}}_{r\rightarrow +\infty} H_n^\star(k_1,\cdots,k_{l-1}, k_l+1,\{1\}^r) =H_n^\star(k_1,\cdots,k_{l-1}, k_l) .        \]
        \end{lem}
        \noindent{\bf Proof:}
$(i)$ Put $a_{n,r}=H_n^\star(\ones^r)$, with $a_{n,0}=1$.
For $n=1$ this equals $1$. For $n\ge2$, appending a last entry
equal to $1$ shows that $a_{n,r}$ is
nondecreasing in $r$, and Lemma~\ref{many1} bounds it by $n$.
Thus $$L_n=\lim_{r\to+\infty}a_{n,r}$$ exists, with $L_1=1$.
Splitting the defining sum according as its first entry equals $n$
or is smaller gives
\[
 a_{n,r}=a_{n-1,r}+\frac{a_{n,r-1}}n
 \qquad(n\ge2,\ r\ge1).
\]
Taking limits yields $(1-1/n)L_n=L_{n-1}$, so induction gives $L_n=n$.

$(ii)$ By definition,
\[
 \begin{split}
 &H_n^\star(k_1,\ldots,k_{l-1},k_l+1,\ones^r)\\
 &\quad=\sum_{n\ge m_1\ge\cdots\ge m_l\ge1}
 \frac{a_{m_l,r}/m_l}{m_1^{k_1}\cdots m_l^{k_l}}
 \longrightarrow H_n^\star(k_1,\ldots,k_l),
 \end{split}
\]
because the outer sum is finite and $(i)$ gives $a_{m_l,r}/m_l\to1$.
$\hfill\Box$\\          
            \begin{prop} \label{or}For $n\geq 2$, there is a total order structure on the set of multiple star  harmonic sums:
     If   $(k_1,\cdots,k_p)\succ (l_1,\cdots,l_q)  $, then
     \[   H_n^\star(k_1,\cdots,k_p)>H_n^\star(l_1,\cdots,l_q).   \]   
     \end{prop}
           \noindent{\bf Proof:} 
For $n\geq 2$, by definition, 
\[
\begin{split}
&\;\;\;\;    H_n^\star(k_1,\cdots,k_r,k_{r+1})        \\
&=  \sum_{n\geq m_1\geq \cdots\geq m_r\geq m_{r+1}\geq 1}    \frac{1}{m_1^{k_1}\cdots m_r^{k_r}m_{r+1}^{k_{r+1}} }           \\
&= \left(\sum_{\substack{n\geq m_1\geq \cdots\geq m_r\geq m_{r+1}\geq 1\\ m_{r+1}=1}} +\sum_{\substack{n\geq m_1\geq \cdots\geq m_r\geq m_{r+1}\geq 1\\ m_{r+1}\geq 2}}     \right)\frac{1}{m_1^{k_1}\cdots m_r^{k_r}m_{r+1}^{k_{r+1}} }           \\
&=H_n^{\star}(k_1,\cdots,k_r)+\sum_{n\geq m_1\geq \cdots\geq m_r\geq m_{r+1}\geq 2}    \frac{1}{m_1^{k_1}\cdots m_r^{k_r}m_{r+1}^{k_{r+1}} }  .   \end{split}
\]
So we have 
\[
 H_n^\star(k_1,\cdots,k_r,k_{r+1})  >H_n^{\star}(k_1,\cdots,k_r).
  \]
  For $$(k_1,\cdots,k_p)\succ (l_1,\cdots,l_q), $$
  we assume that $(k_1,\cdots,k_{i-1})=(l_1,\cdots,l_{i-1})$ and $k_i<l_i$ for some $i\leq \mathrm{min}\{p,q\}$.
  If $q=i$, decreasing the last exponent from $l_i$ to $l_i-1$
  strictly increases the sum, because $n\geq2$ supplies a term with
  $m_i\geq2$. If $q>i$, Lemma \ref{many1} gives the same strict
  inequality as follows:

  \[
  \begin{split}
  &\;\;\;\;  H_n^\star(l_1,\cdots,l_q)     \\
  &\leq  H_n^\star(l_1,\cdots,l_{i-1},l_i,\{1\}^{q-i})        \\
  &\leq \sum_{n\geq m_1\geq \cdots\geq m_i\geq 1} \frac{1}{m_1^{l_1}\cdots m_i^{l_i}} \cdot \sum_{m_i\geq m_{i+1}\geq \cdots \geq m_q\geq 1} \frac{1}{m_{i+1}\cdots m_q}\\
  &<\sum_{n\geq m_1\geq \cdots\geq m_i\geq 1} \frac{m_i}{m_1^{l_1}\cdots m_i^{l_i}} =H_n^\star(l_1,\cdots,l_{i-1},l_i-1).    \\
    \end{split}
  \]
  Since $H_n^\star(l_1,\cdots,l_{i-1},l_i-1)\leq H_n^\star(k_1,\cdots,k_i)\leq H_n^\star(k_1,\cdots,k_p)$,
  we have 
  \[
   H_n^\star(k_1,\cdots,k_p)>H_n^\star(l_1,\cdots,l_q).  \]
 $\hfill\Box$\\    
 
 Now we are ready to prove Theorem \ref{int}.
 For $(k_1,\cdots,k_r,\cdots)\in \widehat{\mathcal{T}}=(\mathbb{Z}^+)^\infty$ and $n\geq 2$, define 
 \[
 a_r=H_n^\star(k_1,\cdots,k_r).
 \]
 By Proposition \ref{or}, we have $a_r<a_{r+1}$. By Lemma \ref{many1}, 
 \[
 a_r\leq H_n^\star(\{1\}^r)<n.
 \]
 As a result, the sequence $\{a_r\,|\,r\geq 1\}$ is an increasing bounded sequence, thus the limit 
 \[
 \mathop{\mathrm{lim}}_{r\rightarrow +\infty} H_n^{\star}(k_1,\cdots,k_r)  \]
 exists and the map 
  \[\mathfrak{h}_n: \widehat{\mathcal{T}}\rightarrow (1, n],
 \]
 \[
  {\bf k}=(k_1,\cdots,k_r,\cdots)\mapsto \mathfrak{h}_n({\bf{k}})=\mathop{\mathrm{lim}}_{r\rightarrow +\infty} H_n^{\star}(k_1,\cdots,k_r).
                \]
                is well-defined. It suffices to show that $\mathfrak{h}_n$ is bijective.
                
                For ${\bf k},{\bf l}\in \widehat{\mathcal{T}}$ and $ {\bf k}\neq {\bf l} $, without loss of generality, we assume that 
                \[
                (k_1,\cdots,k_i)=(l_1,\cdots,l_i), \quad k_{i+1}>l_{i+1}.
                \]
                By Proposition \ref{or}, for $r>i+1$, we have 
              \[ H_n^\star(k_1,\cdots, k_i, k_{i+1},\cdots k_r) <H_n^\star(k_1,\cdots, k_i, k_{i+1}-1) \leq H_n^\star(l_1,\cdots,l_i,l_{i+1}).  \]    
              Therefore
              \[
              \mathfrak{h}_n({\bf{k}})=\mathop{\mathrm{lim}}_{r\rightarrow +\infty} H_n^{\star}(k_1,\cdots,k_r)\leq    H_n^\star(l_1,\cdots,l_i,l_{i+1}).            \]  
              On the other hand 
              \[
              H_n^\star(l_1,\cdots,l_i,l_{i+1})<      \mathfrak{h}_n({\bf{l}})=\mathop{\mathrm{lim}}_{r\rightarrow +\infty} H_n^{\star}(l_1,\cdots,l_r)  .        \]   
            In conclusion, we have     
            \[
              \mathfrak{h}_n({\bf{k}})     <  \mathfrak{h}_n({\bf{l}}).      \]    
            So the map $\mathfrak{h}_n$ is injective.
            
            To show that $\mathfrak{h}_n$ is surjective, for $r\geq 1$, define 
            \[
            Z_{\{1\}^r}(n)=\Big{(}H_n^\star(\{1\}^r), n\Big{]},
            \]
            \[
            Z_{k_1,\cdots,k_r}(n)=\Big{(}H_n^\star(k_1,\cdots,k_r), H_n^\star(k_1,\cdots ,k_{i-1},k_i-1)  \Big{]}
            \]
  for $(k_1,\cdots,k_r)\neq (\{1\}^r)$ and $$k_i\geq 2,k_{i+1}=\cdots=k_r=1.$$
  For each fixed finite prefix $\boldsymbol a$, direct dominated
  convergence in the finite sum gives
  \[
  \lim_{k\to+\infty}H_n^\star(\boldsymbol a,k)
       =H_n^\star(\boldsymbol a),\qquad
  \lim_{k\to+\infty}H_n^\star(k)=1.
  \]
  Here the empty-prefix value is $1$. To verify the partition
  explicitly, write an original cylinder as
  $Z_{\boldsymbol a}(n)=(H_n^\star(\boldsymbol a),U_{\boldsymbol a}]$,
  with $U_\varnothing=n$ for the root cylinder $(1,n]$.
  The definition gives
  \[
  \begin{split}
  Z_{\boldsymbol a,1}(n)
    &=(H_n^\star(\boldsymbol a,1),U_{\boldsymbol a}],\\
  Z_{\boldsymbol a,k}(n)
    &=(H_n^\star(\boldsymbol a,k),
       H_n^\star(\boldsymbol a,k-1)]\qquad(k\ge2).
  \end{split}
  \]
  Strict order makes the lower endpoints strictly decreasing in $k$,
  and their limit is $H_n^\star(\boldsymbol a)$. The sub-intervals are
  therefore pairwise disjoint and exhaust the original interval. A common endpoint
  belongs only to the interval on its left. Iterating this observation
  proves the following partition and refinement identities:

    \[
  Z_{k_1,\cdots,k_r}(n)\bigcap Z_{l_1,\cdots,l_r}(n)=\varnothing,\quad (k_1,\cdots,k_r)\neq (l_1,\cdots,l_r),
  \]
  \[
  (1,n]=\bigcup_{(k_1,\cdots,k_r)\in (\mathbb{Z}^+)^r}  Z_{k_1,\cdots,k_r}(n), \quad r\geq 1, \]
  \[
  Z_{k_1,\cdots,k_r}(n)=\bigcup_{(k_{r+1},\cdots,k_{r+p})\in (\mathbb{Z}^+)^p} Z_{k_1,\cdots,k_r,k_{r+1},\cdots,k_{r+p}}(n)   . \]
  
  \begin{lem}
  For $S\subseteq\mathbb{R}$, denote by $m(S)$ the Lebesgue measure of $S$.
  If $r\geq 1$ and $(k_1,\cdots,k_r)\in (\mathbb{Z}^+)^r$, then
  \[
  m\left(     Z_{k_1,\cdots,k_r}(n)  \right)=\sum_{n\geq m_1\geq \cdots \geq m_r\geq 1}\frac{m_r-1}{m_1^{k_1}\cdots m_r^{k_r}}.
  \]
  \end{lem}
    \noindent{\bf Proof:}   
    By definition, it suffices to show that
    \[
    H_n^\star(k_1,\cdots,k_r)+m\left(     Z_{k_1,\cdots,k_r}(n)  \right)=  \sum_{n\geq m_1\geq \cdots \geq m_r\geq 1}\frac{m_r}{m_1^{k_1}\cdots m_r^{k_r}}.
  \]
  For $(k_1,\cdots,k_r)=(\{1\}^r)$, 
  \[
  \begin{split}
  &\;\;\;\;    \sum_{n\geq m_1\geq \cdots \geq m_r\geq 1}\frac{m_r}{m_1\cdots m_r}     \\
  &=   \sum_{n\geq m_1\geq \cdots \geq m_r\geq 1}\frac{1}{m_1\cdots m_{r-1}}            \\
  &=\sum_{n\geq m_1\geq \cdots \geq m_{r-1}\geq 1}\frac{m_{r-1}}{m_1\cdots m_{r-1}}            \\
  &\qquad\cdots\\
  &=\sum_{n\geq m_1\geq 1}\frac{m_1}{m_1}\\
  &=n\\
  &=    H_n^\star(\{1\}^r)+m\left(     Z_{\{1\}^r}(n)  \right).    
    \end{split}
  \]
   For $(k_1,\cdots,k_r)\neq (\{1\}^r)$ and $$k_i\geq 2,k_{i+1}=\cdots=k_r=1,$$
   similarly, we have 
   \[
   \begin{split}
   &\;\;\;\;    \sum_{n\geq m_1\geq \cdots \geq m_r\geq 1}\frac{m_r}{m_1^{k_1}\cdots  m_i^{k_i} m_{i+1}\cdots m_r}  \\
   &=   \sum_{n\geq m_1\geq \cdots \geq m_{r-1}\geq 1}\frac{m_{r-1}}{m_1^{k_1}\cdots  m_i^{k_i} m_{i+1}\cdots m_{r-1}}            \\
   &\qquad\cdots  \\
   &=  \sum_{n\geq m_1\geq \cdots \geq m_i\geq 1}\frac{m_i}{m_1^{k_1}\cdots  m_i^{k_i} }              \\
   &= H_n^\star(k_1,\cdots,k_r)+m\left(     Z_{k_1,\cdots,k_r}(n)  \right).       \end{split}
   \]
     $\hfill\Box$\\     
    
    For $x\in (1,n]$, if $x=n$, by Lemma \ref{inf1}, one has 
    \[
    x=n=\mathfrak{h}_n(({\{1\}^{\infty}}))   .
     \]
     If $x=H_n^\star(k_1,\cdots,k_r)$, by Lemma \ref{inf1}, one has 
     \[
     x=H_n^\star(k_1,\cdots,k_r)  =\mathfrak{h}_n\left( (k_1,\cdots, k_{r-1}, k_r+1,\{1\}^{\infty}) \right) .  \]
     If $$x\neq H_n^\star(k_1,\cdots,k_r)$$ for all $(k_1,\cdots,k_r)\in \left( \mathbb{Z}^+  \right)^r $, $r\geq 1$, from the basic properties of $$Z_{k_1,\cdots,k_r}(n),$$ there is a 
     \[
     {\bf k}=(k_1,\cdots,k_r,\cdots)\in \widehat{\mathcal{T}}
     \]
     such that
     \[
     x\in Z_{k_1,\cdots,k_r}(n)
     \]
     for all $r\geq 1$.
     As a result,
     \[
              H_n^\star (k_1,\cdots,k_r)  <x\leq  H_n^\star (k_1,\cdots,k_r) +   m\left(     Z_{k_1,\cdots,k_r}(n)  \right).   \]
              Since 
              \[
              \begin{split}
              &\;\;\;\; m\left(     Z_{k_1,\cdots,k_r}(n)  \right)\\
              &= \sum_{n\geq m_1\geq \cdots \geq m_r\geq 1}\frac{m_r-1}{m_1^{k_1}\cdots m_r^{k_r}}
      \\
      &\leq      \sum_{n\geq m_1\geq \cdots \geq m_r\geq 1}\frac{m_r-1}{m_1\cdots m_r}=n-H_n^\star\left(\{1\}^r\right)   ,        \\
              \end{split}              
              \]
              we have 
              \[
              \mathop{\mathrm{lim}}_{r\rightarrow +\infty} m\left(     Z_{k_1,\cdots,k_r}(n)  \right) =   \mathop{\mathrm{lim}}_{r\rightarrow +\infty} \left(  n-H_n^\star\left(\{1\}^r\right)   \right) =0.           \]
              Therefore,
              \[
              x= \mathop{\mathrm{lim}}_{r\rightarrow +\infty} H_n^\star (k_1,\cdots,k_r) =\mathfrak{h}_n({\bf k}).                            \]
  In a word,  the map $\mathfrak{h}_n: \widehat{\mathcal{T}}\rightarrow  (1,n]   $ is bijective.

    \section{Analytic continuation of finite multiple star harmonic sums}\label{ac}
    In this section, we will review the theory of divided differences. For more references, see \cite{boo},  \cite{boo1}.  By using the divided differences, we will give the integral representation of $$H_s^\star(k_1,\cdots,k_r).$$
  As an application,   we will show that $H_s^\star(k_1,\cdots,k_r)$ is well-defined and analytic on $\mathrm{Re}(s)>0$. In this paper, for $t>0$ and $s\in\mathbb{C}$, the function $t^s$ is defined by  
  \[t^s=e^{s\,\mathrm{log}\,t}.\]
  
  \begin{Def}\label{div}
  The divided differences of $f(t)$ relative to a sequence of distinct points $x_0,x_1,x_2,\cdots$ are defined by 
  \[
  [x_0]f=f(x_0), \]
  \[
  [x_0,x_1]f=\frac{[x_1]f-[x_0]f        }{x_1-x_0},
  \]
  \[
  [x_0,x_1,x_2]f=\frac{[x_1,x_2]f-[x_0,x_1]f      }{x_2-x_0}.
  \]
  \[
  \cdots
  \]
  Inductively, the divided difference of order $n$ is given by 
  \[
  [x_0,x_1,\cdots,x_n]f=\sum_{i=0}^n\frac{f(x_i)}{\prod_{\substack{   0\leq j \leq n \\ j\neq i  }    }     (x_i-x_j)     }.
  \]
  By the above explicit formula, 
  \[
  [x_0,x_1,\cdots,x_n]f  \]
  is symmetric with respect to $x_0,x_1,\cdots,x_n$. More precisely, 
  \[
   [x_0,x_1,\cdots,x_n]f  =[x_{\sigma(0)},x_{\sigma(1)},\cdots, x_{\sigma(n)}]f
     \]
     for any permutation $\sigma$ of the set $\{0,1,\cdots,n\}$.
  If some points coincide and $f$ has the required continuous derivatives, the above expression is understood as the continuous extension of divided differences.
  \end{Def}
  
  For the function with  sufficient differentiability, the divided difference has the following geometric explanation, which is known as Hermite-Genocchi formula. For the history of Hermite-Genocchi formula, one can consult the survey \cite{boo1}.
  \begin{Thm}\label{hg} (Hermite-Genocchi) Let $I\subseteq  \mathbb{R}$ be an open interval. For $f\in C^n(I)$ and $x_0,\ldots,x_n\in I$, denote by 
  \[
  S_n=\big{\{}(u_1,\cdots,u_n)\in \mathbb{R}^n\,\big{|}\, u_j\geq 0, \sum_{j=1}^nu_j\leq 1\big{\}}.
  \]
  Then 
  \[
  [x_0,\cdots, x_n]f=\int_{S_n} f^{(n)}(u_0x_0+u_1x_1+\cdots+u_nx_n) du_1\cdots du_n.
  \]
  Here
  \[
  u_0=1-\sum_{j=1}^nu_j.
  \]
  \end{Thm}
   \noindent{\bf Proof:}
   For $n=0$ the formula is evaluation, and for $n=1$ it is the
   fundamental theorem of calculus. Suppose it holds at order $n-1$ and the nodes are distinct.
   Write $u'=(u_1,\ldots,u_{n-1})$ and
   \[
   v=1-\sum_{i=1}^{n-1}u_i,\qquad
   A=\sum_{i=1}^{n-1}u_ix_i.
   \]
   Integrating first over $0\le u_n\le v$ and applying the
   induction hypothesis gives
   \[
   \begin{split}
   &(x_n-x_0)\int_{S_n}
      f^{(n)}\left(\sum_{i=0}^nu_ix_i\right)\,du\\
   &\quad=\int_{S_{n-1}}
      \left[f^{(n-1)}(A+vx_n)-f^{(n-1)}(A+vx_0)\right]\,du'\\
   &\quad=[x_1,\ldots,x_n]f-[x_0,\ldots,x_{n-1}]f.
   \end{split}
   \]
   Division by $x_n-x_0$ proves the formula by the
   divided-difference recurrence. The integral is continuous in
   the nodes, so coincident nodes follow by the continuous
   extension in Definition~\ref{div}.
   $\hfill\Box$\\    
        
  \begin{lem}\label{plus1}
  For a function $g$ and pairwise distinct nodes $1,x_1,\ldots,x_r$, we have 
  \[
  [1,x_1,x_2,\cdots,x_r]\left(tg(t)\right)=\frac{ [1,x_2,\cdots,x_r]g-x_1[x_1,x_2,\cdots,x_r]g     }{1-x_1}.
  \]
  \end{lem}
   \noindent{\bf Proof:}
   The Leibniz formula \cite{boo} is
   \begin{equation}
   [y_0,\cdots,y_m](fg)
   =\sum_{k=0}^m[y_0,\cdots,y_k]f\,
      [y_k,\cdots,y_m]g.
   \label{eq:dd-leibniz}
   \end{equation}
   Taking $f(t)=t$ gives
   \begin{equation}
   [y_0,\cdots,y_m](tg(t))
   =y_0[y_0,\cdots,y_m]g+[y_1,\cdots,y_m]g.
   \label{eq:dd-linear-product}
   \end{equation}
   Symmetry and the divided-difference recurrence give
   \begin{equation}
   [y_0,\cdots,y_m]g
   =\frac{[y_0,y_2,\cdots,y_m]g-[y_1,y_2,\cdots,y_m]g}
          {y_0-y_1}.
   \label{eq:dd-node-recurrence}
   \end{equation}
   Combining the last two identities with
   $(y_0,\ldots,y_m)=(1,x_1,\ldots,x_r)$ proves the assertion.
   $\hfill\Box$\\

  \begin{lem}\label{sint}
   Let $M,r$ be nonnegative integers.
  For arbitrary points $u_0,u_1,\cdots,u_r$, 
  \[[u_0,u_1,\cdots, u_r]t^{M+r}=h_M(u_0,\cdots,u_r).\]
  Here \[
  h_M(u_0,\cdots, u_r)=\sum_{\substack{  a_0+\cdots +a_r=M    \\ a_0,\,\cdots,a_r\geq 0       }} u_0^{a_0}\cdots u_r^{a_r}
  \]
  \end{lem}
    \noindent{\bf Proof:}
    For $u_i\neq u_j, i\neq j$, we have the following explicit formula
    \[
    [u_0,u_1,\cdots, u_r]t^{M+r}=\sum_{i=0}^r \frac{u_i^{M+r}}{\prod_{j\neq i} (u_i-u_j)}
        \]
        Thus 
        \[
        \sum_{M=0}^{+\infty}[u_0,u_1,\cdots, u_r]t^{M+r}\cdot X^M=\sum_{i=0}^r \frac{u_i^{r}}{\prod_{j\neq i} (u_i-u_j)}\cdot \frac{1}{1-u_iX}.                \]
    By the partial-fraction decomposition, it is easy to check that
    \[
    \sum_{i=0}^r \frac{u_i^{r}}{\prod_{j\neq i} (u_i-u_j)}\cdot \frac{1}{1-u_iX}=\prod_{i=0}^r\frac{1}{1-u_iX}.     \]
    As a result, 
    \[
    \begin{split}
    &   \sum_{M=0}^{+\infty}[u_0,u_1,\cdots, u_r]t^{M+r}\cdot X^M        = \prod_{i=0}^r\frac{1}{1-u_iX}               =\sum_{M=0}^{+\infty} h_M(u_0,\cdots,u_r) X^M.
    \end{split}
    \]
    Thus the Lemma is proved for $u_i\neq u_j,i\neq j$. In general cases, the Lemma follows from continuous extension.
  $\hfill\Box$\\    
  
Write
\[
 K_j=k_1+\cdots+k_j,\qquad X_0=1,\qquad
 X_j=x_1\cdots x_{K_j},\qquad dx=dx_1\cdots dx_{K_r}.
\]
For $M\ge1$, the substitution
$$m_j=1+a_j+\cdots+a_r$$ is a bijection from the nonnegative integer tuples
$a_0,\ldots,a_r$, $\sum_{i=0}^ra_i=M-1$, to
$$M\ge m_1\ge\cdots\ge m_r\ge1.$$
Thus Lemma~\ref{sint} and integration over each block give
\begin{equation}
\begin{split}
 &\int_{(0,1)^{K_r}}[1,X_1,\ldots,X_r]t^{M+r-1}\,dx\\
 &\quad=\sum_{M\ge m_1\ge\cdots\ge m_r\ge1}
       \prod_{j=1}^r m_j^{-k_j}
 =H_M^\star(k_1,\ldots,k_r).
\end{split}
\label{eq:integer-star-cube}
\end{equation}
Multiplication by $X_r^s$ adds $s$ to every block exponent,
so for $\Rea s>0$ the same calculation gives
\begin{equation}
\begin{split}
 &\int_{(0,1)^{K_r}}[1,X_1,\ldots,X_r]t^{M+r-1}X_r^s\,dx\\
 &\quad=\sum_{M\ge m_1\ge\cdots\ge m_r\ge1}
       \prod_{j=1}^r(m_j+s)^{-k_j}
 =G_M^\star(s;k_1,\ldots,k_r).
\end{split}
\label{eq:shifted-star-cube}
\end{equation}

  \begin{lem}\label{ker}
  Let  $0<X_r<\cdots <X_1<1$ and $P_j=\prod_{i=1}^j(1-X_i)$, $P_0=1$. Define 
  \[
  A_r(s;X_1,\cdots ,X_r)=[1,X_1,\cdots, X_r]t^{s+r-1}, \;\;\;A_0(s)=1
  \]
  and 
  \[
  \begin{split}
  P_{M,r}(s;X_1,\cdots,X_r)
   &= A_r(M;X_1,\cdots,X_r)\\
   &\quad-\sum_{j=1}^rX_j^sA_j(M;X_1,\cdots,X_j)
     A_{r-j}\left(s;\frac{X_{j+1}}{X_j},\cdots,\frac{X_r}{X_j}\right).
  \end{split} \]
Then \\
$(i)$
For $r\geq 1$, 
\[
A_r(s;X_1,\cdots, X_r)=\frac{1}{1-X_1}A_{r-1}(s; X_2,\cdots,X_r)-\frac{X_1^s}{1-X_1}A_{r-1}\left( s; \frac{X_2}{X_1},\cdots, \frac{X_r}{X_1}    \right);
\]
$(ii)$ In general cases, 
 \[
 A_r(s;X_1,\cdots,X_r)=\frac{1}{P_r}-\sum_{j=1}^r \frac{X_j^s}{P_j}A_{r-j}\left(s; \frac{X_{j+1}}{X_j},\cdots, \frac{X_r}{X_j}      \right);
 \]        
 Equivalently, 
 \[
 A_r(s;X_1,\cdots,X_r)=\frac{1-X_r^s}{P_r}-\sum_{j=1}^{r-1} \frac{X_j^s}{P_j}A_{r-j}\left(s; \frac{X_{j+1}}{X_j},\cdots, \frac{X_r}{X_j}      \right);
 \]
 $(iii)$ \[
 \begin{split}
 P_{M,r}(s;X_1,\cdots,X_r)
  &=A_r(s;X_1,\cdots,X_r)+A_r(M;X_1,\cdots,X_r)\\
  &\quad-A_r(M+s;X_1,\cdots,X_r).
 \end{split}
 \]
 $(iv)$ \[
 \mathop{\mathrm{lim}}_{M\rightarrow +\infty} P_{M,r}(s; X_1,\cdots,X_r)=A_r(s; X_1,\cdots, X_r).
 \]
  \end{lem}
    \noindent{\bf Proof:} 
    $(i)$ By Lemma \ref{plus1},     for $r\geq 1$,
    \[
    [1,X_1,\cdots,X_r]t^{s+r-1}=\frac{  [1,X_2,\cdots ,X_r]t^{s+r-2}-X_1[X_1,X_2,\cdots, X_r]t^{s+r-2}         }{1-X_1}
    \]
    By the definition of divided difference, we have 
    \[
    [X_1,X_2,\cdots, X_r]t^{s+r-2}=X_1^{s-1} \left[1, \frac{X_2}{X_1},\cdots, \frac{X_r}{X_1}      \right] t^{s+r-2}.
    \]
    Thus 
    \[
    \begin{split}
    &\;\;\;\;A_r(s; X_1, \cdots, X_r)            \\
    &=    \frac{1}{1-X_1} [1,X_2,\cdots,X_r]t^{s+r-2}-\frac{X_1^s}{1-X_1}      \left[1, \frac{X_2}{X_1},\cdots, \frac{X_r}{X_1}      \right] t^{s+r-2}          \\
    &=   \frac{1}{1-X_1}A_{r-1}(s; X_2,\cdots,X_r)-\frac{X_1^s}{1-X_1}A_{r-1}\left( s; \frac{X_2}{X_1},\cdots, \frac{X_r}{X_1}    \right).            \\
    \end{split}
    \]
    $(ii)$ Iterate $(i)$ on its first term: the successive denominators
    are $P_1,\ldots,P_r$, and the final term is $A_0/P_r=1/P_r$.
    Isolating the $j=r$ summand gives the equivalent form.
    $(iii)$ For the positive integer $M$, factor $t^M-1$ and apply
    \eqref{eq:dd-linear-product} at the node $1$ to obtain
    \[
    A_r(M+s;X)-A_r(s;X)
       =\sum_{\ell=0}^{M-1}[X_1,\ldots,X_r]t^{s+r-1+\ell},
    \]
    where $X=(X_1,\ldots,X_r)$. Repeated application of the linear
    Leibniz identity \eqref{eq:dd-linear-product} gives
    \[
    \begin{split}
    [X_1,\ldots,X_r]t^{s+r-1+\ell}
       &=\sum_{j=1}^r h_\ell(X_1,\ldots,X_j)X_j^s\\
       &\qquad\cdot A_{r-j}
          \left(s;\frac{X_{j+1}}{X_j},\ldots,\frac{X_r}{X_j}\right).
    \end{split}
    \]
    By Lemma \ref{sint},
    $\sum_{\ell=0}^{M-1}h_\ell(X_1,\ldots,X_j)=A_j(M;X_1,\ldots,X_j)$.
    Substitution and the definition of $P_{M,r}$ prove $(iii)$. \\
    $(iv)$ By the explicit divided-difference formula,
    \[
    A_r(M;X)-A_r(M+s;X)
      =\sum_{j=1}^r
       \frac{X_j^{M+r-1}(1-X_j^s)}
       {(X_j-1)\prod_{\substack{1\le i\le r\\i\ne j}}(X_j-X_i)}
      \longrightarrow0.
    \]
    Indeed, the node $1$ contributes zero, and $0<X_j<1$ for every $j$.
    Part $(iii)$ now gives the assertion.
     $\hfill\Box$\\
     
     \begin{lem}\label{limk}
     For $r\geq 1$, $(k_1,\cdots, k_r)\in \left(\mathbb{Z}^+\right)^r$, and a compact subset 
     \[
     \Omega\subseteq \{s\in \mathbb{C}\,|\, \mathrm{Re}\,s>0\}. 
     \]
      Define 
     \[
     K_r=k_1+\cdots+k_r,
     \]
         \[
X_j=x_1x_2\cdots x_{k_1+\cdots +k_j}, \;\;\; 1\leq j\leq r.
 \]
    There exists a constant  $C_{\Omega,r}$  such that
     \[
     \Big{|}    [1,X_1,\cdots, X_r]t^{s+r-1}  \Big{|}\leq C_{\Omega,r}
     \]
     for all $s\in \Omega$ and  $(x_1,\cdots, x_{K_r})\in (0,1)^{K_r}$. Therefore 
     \[
     I_s(k_1,\cdots,k_r):=\int_{(0,1)^{K_r}} [1,X_1,\cdots ,X_r] t^{s+r-1} dx_1dx_2\cdots dx_{K_r}
     \]
     is holomorphic on $\mathrm{Re}\,s>0$.
        \end{lem}
 \noindent{\bf Proof:}
Put $\sigma_0=\inf_{s\in\Omega}\Rea s>0$,
$\alpha=\min\{1,\sigma_0\}$, and
$C=\sup_{s\in\Omega}|s(s+1)\cdots(s+r-1)|$.
For $V=u_0+u_1X_1+\cdots+u_rX_r$ we have $u_0\le V\le1$,
so the Hermite--Genocchi formula gives
\[
 \begin{split}
 |[1,X_1,\ldots,X_r]t^{s+r-1}|
 &\le C\int_{S_r}u_0^{\alpha-1}\,du_1\cdots du_r\\
 &=\frac{C}{(r-1)!}\int_0^1
       u^{\alpha-1}(1-u)^{r-1}\,du
 \le\frac{C}{\alpha(r-1)!}.
 \end{split}
\]
For each point of the open cube the divided difference is entire in $s$.
The compact-uniform bound, Fubini's theorem and Morera's theorem
prove that $I_s$ is holomorphic.
$\hfill\Box$\\

         \begin{lem}\label{conv} For ${\bf k}=(k_1,\cdots,k_r)\in \left( \mathbb{Z}^+    \right)^r$ and $\Omega$ a compact subset of $$  \{s\in \mathbb{C}\,|\, \mathrm{Re}\,s>0\}.$$
          Define 
     \[
     K_r=k_1+\cdots+k_r,
     \]
         \[
X_j=x_1x_2\cdots x_{k_1+\cdots +k_j}, \;\;\; 1\leq j\leq r.
 \]
         There is a constant $C_{{\bf k},\Omega}$ such that, for all $M\geq 1$ and all $s\in \Omega$, 
         \[
         \begin{split}
         &\int_{(0,1)^{K_r}}
         \big|P_{M,r}(s;X_1,\cdots,X_r)-A_r(s;X_1,\cdots,X_r)\big|\,dx\\
         &\hspace{3em}\leq C_{{\bf k},\Omega}
                    \frac{\log^{r-1}(M+2)}{M+1}.
         \end{split}
         \]
         In particular,
         \[
          \int_{(0,1)^{K_r}}P_{M,r}(s; X_1,\cdots,X_r   )   dx_1\cdots dx_{K_r} \rightarrow       \int_{(0,1)^{K_r}}A_r(s;X_1,\cdots, X_r)    dx_1\cdots dx_{K_r}    \]
          locally uniformly for $\mathrm{Re}\,s>0$.
          \end{lem}
           \noindent{\bf Proof:}
Put
\[
 \sigma_0=\inf_{s\in\Omega}\Rea s>0,\qquad
 B=\sup_{s\in\Omega}|s|,\qquad C=B\max\{1,\sigma_0^{-1}\}.
\]
For $0<t<1$,
\[
 \int_0^1(t+u(1-t))^{\sigma_0-1}\dd u
 =\frac{1-t^{\sigma_0}}{\sigma_0(1-t)}
 \le\max\{1,\sigma_0^{-1}\}.
\]
Thus $1-x^s=s\int_x^1v^{s-1}\dd v$ gives, uniformly for $s\in\Omega$,
\begin{equation}\label{eq:power-difference-bound}
 |1-x^s|\le C(1-x),\qquad 0<x<1.
\end{equation}
The divided-difference identity
\begin{equation}\label{eq:dd-factor-removal}
 [1,y_1,\ldots,y_r]\bigl((t-1)g(t)\bigr)
 =[y_1,\ldots,y_r]g(t)
\end{equation}
and Lemma~\ref{ker}, $(iii)$, give
\[
 P_{M,r}-A_r=-[X_1,\ldots,X_r]G_{M,r},\qquad
 G_{M,r}(t)=t^{M+r-1}\phi_s(t),\quad
 \phi_s(t)=\frac{1-t^s}{1-t}.
\]
Since
\[
 \phi_s^{(j)}(t)
 =s(s-1)\cdots(s-j)
   \int_0^1(1-u)^j(t+u(1-t))^{s-1-j}\dd u,
\]
the preceding integral bound and $t+u(1-t)\ge t$ imply
\begin{equation}\label{eq:phi-derivative-bound}
 |\phi_s^{(j)}(t)|\le C_{\Omega,r}t^{-j},
 \qquad 0<t<1,\quad 0\le j\le r-1.
\end{equation}
Use the simplex notation of Theorem~\ref{hg}, interpreting the
$S_0$ integral as evaluation at its single point. Then
\begin{equation}\label{eq:gm-hermite-genocchi}
 [X_1,\ldots,X_r]G_{M,r}(t)
 =\int_{S_{r-1}}
 G_{M,r}^{(r-1)}(u_0X_1+\cdots+u_{r-1}X_r)\,du.
\end{equation}
Leibniz's rule and \eqref{eq:phi-derivative-bound} give
\begin{equation}\label{eq:gm-derivative-bound}
 \begin{split}
 |G_{M,r}^{(r-1)}(t)|
 &\le C_{\Omega,r}t^M
       \sum_{j=0}^{r-1}\binom{r-1}{j}(M+r-1)^{r-1-j}\\
 &\le C_{\Omega,r}(M+r)^{r-1}t^M.
 \end{split}
\end{equation}
Consequently,
\begin{equation}\label{eq:gm-divided-difference-bound}
 \left|[X_1,\ldots,X_r]G_{M,r}(t)\right|
 \le C_{\Omega,r}(M+r)^{r-1}
 \int_{S_{r-1}}(u_0X_1+\cdots+u_{r-1}X_r)^M\,du.
\end{equation}
Applying Hermite--Genocchi once more, scaling the nodes by $X_1$,
and using \eqref{eq:integer-star-cube}, we obtain
\[
 \begin{split}
 &\int_{(0,1)^{K_r}}\int_{S_{r-1}}
       (u_0X_1+\cdots+u_{r-1}X_r)^M\,du\,dx\\
 &\quad=\frac{M!}{(M+r-1)!}
       \frac{H_{M+1}^\star(k_2,\ldots,k_r)}{(M+1)^{k_1}}
 \le\frac{(2\log(M+2))^{r-1}}{(M+1)^r}.
 \end{split}
\]
The last bound uses $k_1\ge1$ and bounds the tail sum by
$(H_{M+1}^\star(1))^{r-1}$, with empty-index value $1$.
Integrating \eqref{eq:gm-divided-difference-bound} and applying this bound
now yields
\[
 \int_{(0,1)^{K_r}}|P_{M,r}-A_r|\,dx
 \le C_{{\bf k},\Omega}\frac{\log^{r-1}(M+2)}{M+1}.
\]
This proves the assertion, including local uniform convergence.
$\hfill\Box$\\

  \begin{Thm}
  For  $(k_1,\cdots,k_r)\in (\mathbb{Z}^+)^r$. The function $$H_s^\star(k_1,\cdots,k_r)$$ is well-defined and analytic on $\mathrm{Re}\,(s)>0$.
 Furthermore, let
 \[
 X_0=1,X_j=x_1x_2\cdots x_{k_1+\cdots +k_j}, \;\;\; 1\leq j\leq r.
 \]
 Then one has the following integral representation
 \[
 H_s^\star(k_1,\cdots,k_r)=\int_{(0,1)^{k_1+\cdots+k_r}}[X_0,X_1,\cdots,X_r]t^{s+r-1}dx_1dx_2\cdots dx_{k_1+\cdots+k_r}
  \]
  on $\mathrm{Re}(s)>0$ .  \end{Thm}
  \noindent{\bf Proof:}
Use $I_s$ from Lemma~\ref{limk}, with $I_s(\varnothing)=1$.
The ratios $X_{j+\ell}/X_j$ involve only the variables after
the $j$th block. Thus \eqref{eq:integer-star-cube},
\eqref{eq:shifted-star-cube}, and the definition of $P_{M,r}$ give
\[
 \begin{split}
 &H_M^\star(k_1,\ldots,k_r)\\
 &\quad-\sum_{j=1}^rG_M^\star(s;k_1,\ldots,k_j)
                  I_s(k_{j+1},\ldots,k_r)\\
 &\qquad=\int_{(0,1)^{K_r}}P_{M,r}(s;X_1,\ldots,X_r)\,dx
 \longrightarrow I_s(k_1,\ldots,k_r).
 \end{split}
\]
Lemma~\ref{limk} and $|X_j^s|\le1$ justify the absolute
Fubini interchanges, and Lemma~\ref{conv} gives the convergence
locally uniformly on $\Rea s>0$.
Every tail in the sum has smaller depth. Induction from the
empty index therefore identifies this limit with the defining
recursion for $H_s^\star$, including depth one.
Consequently $H_s^\star=I_s$, which is holomorphic by
Lemma~\ref{limk}. At depth one the difference series is absolutely
convergent, since $|n^{-k}-(n+s)^{-k}|\le k|s|n^{-k-1}$.
$\hfill\Box$\\
  
  \section{Finite zeta-star correspondence in general cases}
  
  In this section,  we will give a complete proof of Theorem \ref{com}. Many basic properties of finite zeta-star correspondence are also discussed in this section.
  \begin{lem}\label{dom}
  Let $\Omega$ be a compact subset of $\{s:\Rea s>0\}$ and put
  $\sigma_0=\inf_{s\in\Omega}\Rea s$.
  For every $0<\epsilon<\min\{1,\sigma_0\}$ there is $C_{\Omega,\epsilon}>0$
  such that, with $D_{\Omega,\epsilon}=\min\{2,1+\sigma_0\}-\epsilon>1$,
  \[
  \left|H_s^\star(k_1,\ldots,k_r,k_{r+1})
             -H_s^\star(k_1,\ldots,k_r)\right|
      \leq C_{\Omega,\epsilon} D_{\Omega,\epsilon}^{-r}
  \]
  for every $s\in\Omega$, $r\geq1$ and every positive integer index.
  In particular, the assertion holds for each fixed $s$.
  \end{lem}
  \noindent{\bf Proof:}
Write $K_j=k_1+\cdots+k_j$, $X_j=x_1\cdots x_{K_j}$ and
$f(t)=t^{s+r-1}$. By Lemma \ref{plus1} and symmetry,
\[
 \begin{split}
 &H_s^\star(k_1,\ldots,k_r,k_{r+1})-H_s^\star(k_1,\ldots,k_r)\\
 &=\int_{(0,1)^{K_{r+1}}}\frac{X_{r+1}}{1-X_{r+1}}
 \bigl([1,X_1,\ldots,X_r]f
      -[X_1,\ldots,X_r,X_{r+1}]f\bigr)\,dx.
 \end{split}
\]
Put $\alpha=\min\{1,\sigma_0\}$, $B=\sup_{s\in\Omega}|s|$, and
$(s)_r=s(s+1)\cdots(s+r-1)$.
Every convex combination $V$ of either set of nodes lies in
$[X_{r+1},1]$, so
$|V^{s-1}|\le V^{\alpha-1}\le X_{r+1}^{\alpha-1}$.
Moreover,
\[
 \frac{|(s)_r|}{r!}
 \le\prod_{j=1}^r\left(1+\frac Bj\right)\le e^B r^B.
\]
Hermite--Genocchi and $\operatorname{vol}(S_r)=1/r!$ therefore
bound the parentheses above by
$C_\Omega r^B X_{r+1}^{\alpha-1}$.
Expanding $(1-X_{r+1})^{-1}$, integrating nonnegative terms,
and using $K_{r+1}\ge r+1$ therefore give
\[
 \begin{split}
 &\left|H_s^\star(k_1,\ldots,k_r,k_{r+1})
             -H_s^\star(k_1,\ldots,k_r)\right|\\
 &\le C_\Omega r^B\sum_{m\ge1}(m+\alpha)^{-r-1}\\
 &\le C_\Omega r^B
       \left(1+\frac{1+\alpha}{r}\right)(1+\alpha)^{-r-1}.
 \end{split}
\]
Since $1+\alpha=\min\{2,1+\sigma_0\}$, absorbing the polynomial
factor into $D_{\Omega,\epsilon}^{-r}$ proves the result.
$\hfill\Box$\\

     By Lemma \ref{dom}, for $\mathrm{Re}\,(s)>0$ and ${\bf k}=(k_1,\cdots,k_r,\cdots)$, the 
     sequence 
     \[
     H_s^\star(k_1,\cdots,k_r), r\geq 1
     \]
     is a Cauchy sequence in $\mathbb{C}$. So the limit
     \[
     \mathop{\mathrm{lim}}_{r\rightarrow +\infty} H_s^\star(k_1,\cdots,k_r)     \]
     exists, locally uniformly in $s$ and uniformly over all infinite
     indices on each compact subset. More explicitly, for the constants
     of Lemma~\ref{dom} and every $r\geq1$,
     \[
     \sup_{s\in\Omega}\sup_{{\bf k}\in\widehat{\mathcal T}}
       \left|\mathfrak h_s({\bf k})
                    -H_s^\star(k_1,\ldots,k_r)\right|
       \leq\frac{C_{\Omega,\epsilon}D_{\Omega,\epsilon}^{-r}}
                    {1-D_{\Omega,\epsilon}^{-1}}.
     \]
     This follows by summing the extension bounds from depth $r$ onward.
     In particular, each $\mathfrak h_s({\bf k})$ is holomorphic in $s$
     on $\Rea s>0$. Thus Theorem~\ref{com}, $(i)$, is proved.

For the proof of Theorem \ref{com}, $(ii)$, we use the total order structure on 
$$\widehat{\mathcal{S}}=\Big{\{}(k_1,\cdots,k_r)\,\Big{|}\,k_i\in\mathbb{Z}^{+}\ (1\leq i\leq r), r\geq 1\Big{\}}$$
which is introduced  in Section \ref{fzs}.

   To deal with the order structure of $H_s^\star(k_1,\cdots,k_r)$ for $s>1$, define 
 \[
 \left(\mathcal{F}_kh\right)(z):=\int_{(0,z)\times (0,1)^{k-1}} \frac{h(1)-h( u_1\cdots u_k )     }{1-u_1\cdots u_k} du_1\cdots du_k, \quad 0\leq z\leq 1,\quad k\in \mathbb{Z}^+
 \]
 whenever the displayed integral exists absolutely.
 For $$(k_1,\cdots,k_r)\in \left(\mathbb{Z}^+\right)^r$$ and  a function $h$, let
 \[
 \mathcal{L}_{\varnothing}(h)=h(1),
 \]
 \[
 \mathcal{L}_{k_1,\cdots,k_r}(h):=   \int_{(0,1)^{K_r}} [1,X_1,\cdots,X_r] \left( t^{r-1}h(t)\right)dx_1\cdots dx_{K_r},
 \]
 which 
 \[
 K_j=k_1+\cdots +k_j,\qquad X_j=x_1x_2\cdots x_{K_j}, \qquad 1\leq j\leq r.
 \]
 It is clear that
 \[
 H_s^\star(k_1,\cdots,k_r)=\mathcal{L}_{k_1,\cdots,k_r}(t^s).
 \]
 
 \begin{lem}\label{op}
 For $r\geq 1$, whenever the displayed integrals are absolutely convergent, one has 
 \[
 \mathcal{L}_{k_1,\cdots,k_r}(h)=\mathcal{L}_{k_2,\cdots,k_r}\left( \mathcal{F}_{k_1}(h)   \right).
 \]
 Thus 
 \[
 H_s^\star(k_1,\cdots,k_r)=\left(\mathcal{F}_{k_r}\cdots      \mathcal{F}_{k_1} t^s\right)(1)
 \]
 \end{lem}
   \noindent{\bf Proof:}
Put $$Q_h(t)=(h(1)-h(t))/(1-t),\quad X=X_1,$$ and
$z_i=X_{i+1}/X_1$ for $0\le i\le r-1$, so $z_0=1$.
An $r$-th divided difference annihilates $t^{r-1}h(1)$.
Thus \eqref{eq:dd-factor-removal} and scaling the remaining nodes give
\[
\begin{split}
 &\;\;\;\;[1,Xz_0,\ldots,Xz_{r-1}]t^{r-1}h(t)\\
 &=[1,Xz_0,\ldots,Xz_{r-1}] \left(t^{r-1}(h(t)-h(1)) \right) \\
 &=[z_0,\ldots,z_{r-1}]t^{r-1}Q_h(Xt).
 \end{split}
\]
For $0<v\le1$, substitution in the first integration variable gives
\[
 \int_{(0,1)^{k_1}}Q_h(vX)\,dx
 =\frac{(\mathcal F_{k_1}h)(v)}v.
\]
Under the stated absolute convergence assumptions, the finite
divided difference commutes with this integration. Hence
\[
 \begin{split}
 &\int_{(0,1)^{k_1}}
 [1,Xz_0,\ldots,Xz_{r-1}]t^{r-1}h(t)\,dx\\
 &\quad=[z_0,\ldots,z_{r-1}]
       \bigl(t^{r-2}(\mathcal F_{k_1}h)(t)\bigr).
 \end{split}
\]
Integrating the remaining variables proves the first identity;
for $r=1$, the last difference is simply
$(\mathcal F_{k_1}h)(1)=\mathcal L_\varnothing(\mathcal F_{k_1}h)$.

For $h=t^s$, \eqref{eq:power-difference-bound} makes $Q_h$ bounded.
Thus $\mathcal F_kh$ is Lipschitz, and
$|Q_g|\le\operatorname{Lip}(g)$ makes all subsequent iterates
well-defined. Lemma~\ref{limk} gives initial absolute integrability;
the triangle inequality after each block integration preserves it.
Iteration proves the second identity.
$\hfill\Box$\\
   
\begin{lem}\label{e1}
For $\mathrm{Re}\,(s)>0$ and every integer $m\geq0$, one has
\[
H_s^\star(\{1\}^m)=\frac{s}{m!}\int^1_0\left( -\mathrm{log}\;t \right)^m(1-t)^{s-1}dt.
\]
\end{lem}
 \noindent{\bf Proof:}
Put $T=\mathcal F_1$ and $\ell(u)=-\log(1-u)$.
Let $\varphi$ be absolutely continuous, with $\varphi(0)=0$ and
$\int_0^1|\varphi'(v)|\ell(v)^m\dd v<\infty$ for every integer $m\ge0$.
Since
\[
 (T\varphi)'(u)=\frac1{1-u}\int_u^1\varphi'(v)\dd v,
\]
Tonelli gives
\[
 \begin{split}
 \int_0^1|(T\varphi)'(u)|\ell(u)^m\dd u
 &\le\int_0^1|\varphi'(v)|
       \int_0^v\frac{\ell(u)^m}{1-u}\dd u\,\dd v\\
 &=\frac1{m+1}\int_0^1|\varphi'(v)|\ell(v)^{m+1}\dd v.
 \end{split}
\]
Thus $T$ preserves this class. The same calculation without
absolute values, starting with the fundamental theorem of calculus,
proves by induction that
\[
 (T^m\varphi)(1)=\frac1{m!}\int_0^1\varphi'(v)\ell(v)^m\dd v.
\]
This applies to $\varphi(v)=v^s$ for $\Rea s>0$:
all absolute moments are finite, since $\ell(v)=O(v)$ at zero
and only logarithmic powers occur at the other endpoint.
Lemma~\ref{op} and the substitution $v=1-t$ now give the assertion.
 $\hfill\Box$\\
 
    \begin{lem}\label{mont}
    For a real-valued function $h\in C^1[0,1]$, define 
    \[
    Q_h(z)=\frac{h(1)-h(z)}{1-z},\qquad 0<z<1.
    \]
    If $Q_h$ is positive and  strictly increasing on $(0,1)$, then for $k\in \mathbb{Z}^+$, one has
    \[
    \frac{(\mathcal{F}_kh)(z)}{z}<(\mathcal{F}_kh)(1), \qquad 0<z<1.
    \]
    Moreover, $Q_{\mathcal{F}_kh}$ is positive and strictly increasing on $(0,1)$.
     \end{lem}
     \noindent{\bf Proof:}
     By changing of variables, we have 
     \[
     \frac{  ( \mathcal{F}_kh   )(z)  }{z}=\int_{(0,1)^k} Q_h(zu_1\cdots u_k) du_1\cdots du_k
     \]
     and 
     \[
  ( \mathcal{F}_kh   )(1)  =\int_{(0,1)^k} Q_h(u_1\cdots u_k) du_1\cdots du_k. 
     \]
     For $z,u_1,\cdots,u_k\in (0,1)$, one has $$zu_1\cdots u_k<u_1\cdots u_k.$$
     From the fact that $Q_h$ is strictly increasing, we have 
     \[
     Q_h(zu_1\cdots u_k)<Q_h(u_1\cdots u_k).
     \]
     Thus 
      \[
    \frac{(\mathcal{F}_kh)(z)}{z}<(\mathcal{F}_kh)(1), \qquad 0<z<1.
    \]
    
    Denote by $\xi=\mathcal{F}_kh$. If $k=1$, then 
    \[
    \xi^\prime(z)=Q_h(z).
    \]
    Thus $\xi^\prime$ is strictly increasing. If $k\geq 2$, by definition 
    \[
      \xi^\prime(z)=\int^1_0Q_h(zv_1\cdots v_{k-1}   )dv_1\cdots dv_{k-1} .   \]
      As $Q_h$ is strictly increasing, $\xi^\prime$ is also strictly increasing. In both cases, the function $\xi$ is a strictly convex function and $Q_\xi(z)>0$ for $z\in (0,1)$. Then for $1>z>z^\prime>0$, 
      \[\begin{split}
      &\;\;\;\; Q_{\xi}(z)-Q_{\xi}(z^\prime)\\
      &=\frac{ \xi(1)-\xi(z)   }{1-z}-\frac{\xi(1)-\xi(z^\prime) }{1-z^\prime}                        \\
      &=    \frac{1}{(1-z)(1-z^\prime)}  \left[ (z-z^\prime) \xi(1) +(1-z)\xi(z^\prime)-(1-z^\prime) \xi(z)    \right]       \\
      &=  \frac{1}{1-z}  \left[ \frac{z-z^\prime}{1-z^\prime} \xi(1) +\frac{1-z}{1-z^\prime}\xi(z^\prime)- \xi(z)    \right]   >0  .  \\      \end{split}
      \]   
     In conclusion,  $Q_{\mathcal{F}_kh}$ is positive and strictly increasing on $(0,1)$.      $\hfill\Box$\\

 \begin{lem}\label{rea1}
     For $s\in \mathbb{R}$, $s>1$ and $m\geq 0$, one has \\
     $(i)$
     \[
     \begin{split}
     & H_s^\star(\{1\}^m)        < s.         \\
     \end{split}
     \]
     $(ii)$ 
     \[
     H_s^\star(k_1,\cdots,k_{r-1}, k_r+1,\{1\}^m)<H_s^\star(k_1,\cdots,k_{r-1}, k_r).     \]
     $(iii)$ For $i_1,\cdots,i_r\geq 0$ and $(i_1,\cdots,i_r)\neq (0,\cdots,0)$, 
     \[
     H_s^\star(k_1+i_1,\cdots,k_{r-1}+i_{r-1}, k_r+i_r)<H_s^\star(k_1,\cdots,k_{r-1}, k_r).  
          \]
     \end{lem}
     \noindent{\bf Proof:} $(i)$ By Lemma \ref{e1}, for $s>1$, one has
     \[
     0<(1-t)^{s-1}<1,\qquad t\in (0,1).
     \]
     So we have 
     \[
     H_s^\star(\{1\}^m)<\frac{s}{m!}\int^1_0\left( -\mathrm{log}\;t \right)^mdt.
          \]
          As 
          \[
          \int^1_0\left( -\mathrm{log}\;t \right)^mdt=\int^{+\infty}_0 u^m e^{-u}du=\Gamma(m+1)=m!,
          \]     
          one has $$H_s^\star(\{1\}^m)< s.$$
          $(ii)$ For fixed $s>1$, define $h(t)=t^s, t\in[0,1]$. 
          Since 
          \[
          Q_h(z)=\frac{ 1-z^s }{1-z}=s\int^1_0 \left(z+(1-z)u     \right)^{s-1} du,
          \]
          $Q_h(z)$ is positive and  strictly increasing on $0<z<1$.
          Define 
          \[
          g=\mathcal{F}_{k_{r-1}}\cdots \mathcal{F}_{k_1} t^s.
          \]
          For $r=1$, $g$ is interpreted as $g=t^s$. By Lemma \ref{mont}, 
          \[
          Q_g(z)=\frac{g(1)-g(z)   }{1-z}
          \]
          is positive and  strictly increasing on $(0,1)$.
          
          Define $F=\mathcal{F}_{k_r} g$. By Lemma \ref{op}, we have 
          \begin{equation}
          F(1)=\left(\mathcal{F}_{k_r} g \right)(1)=H_s^\star(k_1,\cdots,k_{r-1},k_r)\label{eq:operator-prefix-endpoint}
          \end{equation}
          and 
          \begin{equation}
          H_s^\star(k_1,\cdots,k_{r-1}, k_r+1,\{1\}^m)= \left(\mathcal{F}_1^m \mathcal{F}_{k_r+1}g    \right)(1).\label{eq:operator-ones-tail}
          \end{equation}
          
          Apply the general identity proved in Lemma~\ref{e1} to
          $\varphi=\mathcal F_{k_r+1}g$, which vanishes at $0$.
          Differentiating its defining integral gives
          \[
          \varphi'(u)=\int_{(0,1)^{k_r}}
          \frac{g(1)-g(uv_1\cdots v_{k_r})}{1-uv_1\cdots v_{k_r}}\,dv
             =\frac{(\mathcal F_{k_r}g)(u)}u=\frac{F(u)}u.
          \]
          So we have 
          \[
           H_s^\star(k_1,\cdots,k_{r-1}, k_r+1,\{1\}^m)= \left(\mathcal{F}_1^m \mathcal{F}_{k_r+1}g    \right)(1)=\frac{1}{m!}\int^1_0\frac{F(u)}{u}   \mathrm{log}^{m}\frac{1}{1-u} du.                         \]
          By Lemma \ref{mont}, one has 
          \[
          \frac{F(u)}{u}<F(1),\qquad 0<u<1.
          \]
          Therefore, 
          \[
          H_s^\star(k_1,\cdots,k_{r-1}, k_r+1,\{1\}^m) <\frac{F(1)}{m!}   \int^1_0  \mathrm{log}^{m}\frac{1}{1-u} du.                   \]  
          Since 
          \[
          \int^1_0  \mathrm{log}^{m}\frac{1}{1-u} du= \int^1_0  \mathrm{log}^{m}\frac{1}{u} du=\int^{+\infty}_0 x^m e^{-x}dx=m!                \]   
          and $F(1)=H_s^\star(k_1,\cdots,k_{r-1},k_r)$, we have 
          $$H_s^\star(k_1,\cdots,k_{r-1}, k_r+1,\{1\}^m)<   H_s^\star(k_1,\cdots,k_{r-1},k_r).$$
          $(iii)$ By Hermite--Genocchi, the kernel
\[
 [1,X_1,\ldots,X_r]t^{s+r-1}
 =\frac{\Gamma(s+r)}{\Gamma(s)}
   \int_{S_r}\left(u_0+\sum_{j=1}^r u_jX_j\right)^{s-1}\,du
\]
is strictly increasing in each node $X_j$ when $s>1$.
Let $q_j$ be the product of the extra cube variables inserted
in the $j$th block, with $q_j=1$ when $i_j=0$.
The new nodes are $X_\ell q_1\cdots q_\ell$.
None increases, and at least one strictly decreases because
$(i_1,\ldots,i_r)\ne(0,\ldots,0)$.
The cube-integral representation, integrated over the original
and extra variables, therefore proves $(iii)$.
$\hfill\Box$\\
     
\begin{lem}\label{rlim}
     For $s>1$, one has \\
     $(i)$
     \[
         \mathop{\mathrm{lim}}_{m\rightarrow +\infty} H_s^\star(\{1\}^m)= s;         \]
     $(ii)$ 
     \[
     \mathop{\mathrm{lim}}_{m\rightarrow +\infty} H_s^\star(k_1,\cdots,k_{r-1}, k_r+1,\{1\}^m)=H_s^\star(k_1,\cdots,k_{r-1}, k_r).     \]
\end{lem}
\noindent{\bf Proof:}
$(i)$ For $s>1$, Bernoulli's inequality gives
$0\le1-(1-u)^{s-1}\le\max\{1,s-1\}u$.
Lemma~\ref{e1} therefore yields
\[
 \left|H_s^\star(\ones^m)-s\right|
 \le\frac{s\max\{1,s-1\}}{m!}
       \int_0^1u\log^m(1/u)\dd u
 =\frac{s\max\{1,s-1\}}{2^{m+1}}\rightarrow0.
\]
$(ii)$ The following argument also applies to complex $s$ with
$\Rea s\ge1$. The function $t^s$ is Lipschitz on $[0,1]$ with
constant at most $|s|$, and $\mathcal F_k$ does not increase this
constant, since $\|Q_h\|_\infty\le\operatorname{Lip}(h)$.
Its first image is $C^1[0,1]$, because $Q_{t^s}$ extends
continuously to the endpoints with values $1$ and $s$; all subsequent
images are also $C^1$. Thus, for
$F=\mathcal F_{k_r}\cdots\mathcal F_{k_1}t^s$, we have
$F(0)=0$, $\operatorname{Lip}(F)\le|s|$, and
\[
 \left|\frac{F(v)}v-F(1)\right|
 \le(1-v)\frac{|F(v)|}{v}+|F(1)-F(v)|
 \le2|s|(1-v).
\]
The integral identity in the proof of Lemma~\ref{rea1}, $(ii)$,
valid also for complex-valued functions, gives
\[
 \begin{split}
 &\left|H_s^\star(k_1,\ldots,k_{r-1},k_r+1,\ones^m)
              -H_s^\star(k_1,\ldots,k_r)\right|\\
 &\quad\le\frac{2|s|}{m!}\int_0^1u\log^m(1/u)\dd u
 =|s|\,2^{-m}\rightarrow0.
 \end{split}
\]
          $\hfill\Box$\\     

\begin{Thm}\label{rord}
For $s>1$, there is a total order structure on the set of generalized multiple star  harmonic sums:
     If   $(k_1,\cdots,k_p)\succ (l_1,\cdots,l_q)  $, then
     \[   H_s^\star(k_1,\cdots,k_p)>H_s^\star(l_1,\cdots,l_q).   \]   
\end{Thm}
\noindent{\bf Proof:}
The integral identity in the proof of Lemma~\ref{dom} and
Hermite--Genocchi show that every one-term extension of a nonempty index
strictly increases $H_s^\star$ when $s>1$: the two divided differences
have the same common nodes, while the remaining node decreases
from $1$ to $X_{p+1}<1$, and the integrand $t^{s-1}$ is strictly
increasing.

If the first unequal entries are $a<b$, write the indices as
$(\mathbf u,a,\boldsymbol\alpha)$ and
$(\mathbf u,b,\boldsymbol\beta)$, and put
$m=|\boldsymbol\beta|$. Extension monotonicity and
Lemma~\ref{rea1}, $(ii),(iii)$, give
\[
 \begin{split}
 H_s^\star(\mathbf u,a,\boldsymbol\alpha)
 &\ge H_s^\star(\mathbf u,a)
 >H_s^\star(\mathbf u,a+1,\{1\}^m)\\
 &\ge H_s^\star(\mathbf u,b,\boldsymbol\beta).
 \end{split}
\]
The prefix case follows from strict extension monotonicity.
$\hfill\Box$\\
  
   Now we are ready to prove Theorem \ref{com}, $(ii)$.
   For $s>1$ and every $k\geq1$, the depth-one integrand satisfies
   $(1-z^s)/(1-z)>1$ on $(0,1)$; hence $H_s^\star(k)>1$.
   Strict increase under extensions therefore gives the lower endpoint
   of the range $(1,s]$.
 For $$(k_1,\cdots,k_r,\cdots)\in \widehat{\mathcal{T}}=(\mathbb{Z}^+)^\infty$$ and $s>1$, define 
 \[
 a_r=H_s^\star(k_1,\cdots,k_r).
 \]
 By Theorem \ref{rord}, we have $a_r<a_{r+1}$. By Lemma \ref{rea1}, 
 \[
 a_r\leq H_s^\star(\{1\}^r)<s.
 \]
 As a result, the sequence $\{a_r\,|\,r\geq 1\}$ is an increasing bounded sequence, thus the limit 
 \[
 \mathop{\mathrm{lim}}_{r\rightarrow +\infty} H_s^{\star}(k_1,\cdots,k_r)  \]
 exists and the map 
  \[\mathfrak{h}_s: \widehat{\mathcal{T}}\rightarrow (1, s],
 \]
 \[
  {\bf k}=(k_1,\cdots,k_r,\cdots)\mapsto \mathfrak{h}_s({\bf{k}})=\mathop{\mathrm{lim}}_{r\rightarrow +\infty} H_s^{\star}(k_1,\cdots,k_r).
                \]
                is well-defined. It suffices to show that $\mathfrak{h}_s$ is bijective.  
                
 For ${\bf k},{\bf l}\in \widehat{\mathcal{T}}$ and $ {\bf k}\neq {\bf l} $, without loss of generality, we assume that 
                \[
                (k_1,\cdots,k_i)=(l_1,\cdots,l_i), \quad k_{i+1}>l_{i+1}.
                \]
                By Lemma \ref{rea1}, for $r>i+1$, we have 
              \[ H_s^\star(k_1,\cdots, k_i, k_{i+1},\cdots k_r) <H_s^\star(k_1,\cdots, k_i, k_{i+1}-1) \leq H_s^\star(l_1,\cdots,l_i,l_{i+1}).  \]    
              Therefore
              \[
              \mathfrak{h}_s({\bf{k}})=\mathop{\mathrm{lim}}_{r\rightarrow +\infty} H_s^{\star}(k_1,\cdots,k_r)\leq    H_s^\star(l_1,\cdots,l_i,l_{i+1}).            \]  
              On the other hand 
              \[
              H_s^\star(l_1,\cdots,l_i,l_{i+1})<      \mathfrak{h}_s({\bf{l}})=\mathop{\mathrm{lim}}_{r\rightarrow +\infty} H_s^{\star}(l_1,\cdots,l_r)  .        \]   
            In conclusion, we have     
            \[
              \mathfrak{h}_s({\bf{k}})     <  \mathfrak{h}_s({\bf{l}}).      \]    
            So the map $\mathfrak{h}_s$ is injective.

To show that $\mathfrak{h}_s$ is surjective, for $r\geq 1$, define 
            \[
            Z_{\{1\}^r}(s)=\Big{(}H_s^\star(\{1\}^r), s\Big{]},
            \]
            \[
            Z_{k_1,\cdots,k_r}(s)=\Big{(}H_s^\star(k_1,\cdots,k_r), H_s^\star(k_1,\cdots ,k_{i-1},k_i-1)  \Big{]}
            \]
  for $(k_1,\cdots,k_r)\neq (\{1\}^r)$ and $$k_i\geq 2,k_{i+1}=\cdots=k_r=1.$$
  The partition assertion also needs the endpoint limits
  \[
  \lim_{k\to+\infty}H_s^\star(\boldsymbol a,k)
       =H_s^\star(\boldsymbol a),\qquad
  \lim_{k\to+\infty}H_s^\star(k)=1.
  \]
  To see them, let $h=\mathcal F_{a_p}\cdots\mathcal F_{a_1}t^s$
  (or $h=t^s$ for the empty prefix). Then $h(0)=0$ and $Q_h$ is
  continuous and bounded on $[0,1]$. Since
  \[
  (\mathcal F_kh)(1)=\int_{(0,1)^k}Q_h(u_1\cdots u_k)\,du,
  \qquad Q_h(0)=h(1),
  \]
  these limits follow from continuity  of $Q_h$ on $[0,1]$.
  Write $Z_{\boldsymbol a}(s)
  =(H_s^\star(\boldsymbol a),U_{\boldsymbol a}]$, with
  $U_\varnothing=s$ for the root cylinder $(1,s]$.
  As in the integer case, its sub-intervals are exactly
  \[
  \begin{split}
  Z_{\boldsymbol a,1}(s)
    &=(H_s^\star(\boldsymbol a,1),U_{\boldsymbol a}],\\
  Z_{\boldsymbol a,k}(s)
    &=(H_s^\star(\boldsymbol a,k),
       H_s^\star(\boldsymbol a,k-1)]\qquad(k\ge2).
  \end{split}
  \]
  Their lower endpoints strictly decrease to
  $H_s^\star(\boldsymbol a)$. Thus the sub-intervals are disjoint and
  exhaust the original interval, with each common endpoint assigned to exactly
  one sub-interval. Iteration proves

    \[
  Z_{k_1,\cdots,k_r}(s)\bigcap Z_{l_1,\cdots,l_r}(s)=\varnothing,\quad (k_1,\cdots,k_r)\neq (l_1,\cdots,l_r),
  \]
  \[
  (1,s]=\bigcup_{(k_1,\cdots,k_r)\in (\mathbb{Z}^+)^r}  Z_{k_1,\cdots,k_r}(s), \quad r\geq 1, \]
  \[
  Z_{k_1,\cdots,k_r}(s)=\bigcup_{(k_{r+1},\cdots,k_{r+p})\in (\mathbb{Z}^+)^p} Z_{k_1,\cdots,k_r,k_{r+1},\cdots,k_{r+p}}(s)   . \]
  \begin{lem}\label{eqp}
  For $m\geq 0$, let $c>0, a_0,\cdots,a_m>0, \alpha>0$. Then 
  \[
  \int^1_0[a_0,\cdots,a_m,cx,cx]t^{\alpha+1} dx=[a_0,\cdots,a_m, c]t^\alpha.
  \]
  \end{lem}
   \noindent{\bf Proof:}
   Define 
   \[
   f(x)=[a_0,\cdots,a_m,cx]t^{\alpha+1}.
   \]
   Near $x=0$, the fixed positive nodes are separated from $cx$.
   The explicit divided-difference formula and its confluent limits
   express $f(x)$ through smooth rational coefficients and
   $(cx)^{\alpha+1}$. This power is $C^1$ at $0$ because $\alpha>0$.
   Positive coincident nodes are covered by smooth differentials.
   Hence $f\in C^1[0,1]$, justifying the fundamental theorem of calculus
   below, including the endpoint $x=0$.
   By definition,
   \[\begin{split}
  &\;\;\;\; [a_0,\cdots,a_m,cx,cx]t^{\alpha+1} \\
  &=\mathop{\mathrm{lim}}_{\varepsilon\rightarrow\, 0} \frac{ [a_0,\cdots,a_m,c(x+\varepsilon)]t^{\alpha+1}- [a_0,\cdots,a_m,cx]t^{\alpha+1}  }{c(x+\varepsilon)-cx}\\
  &=\frac{ f^\prime(x) }{c}.
   \end{split}
      \]
      Thus  \[
      \begin{split}
  &\;\;\;\;\int^1_0[a_0,\cdots,a_m,cx,cx]t^{\alpha+1} dx\\
  &=\frac{1}{c}\int^1_0 f^\prime(x) dx    \\
  &=\frac{1}{c}\left(f(1)-f(0)      \right)\\
  &=\frac{1}{c}\left([a_0,\cdots, a_m,c]t^{\alpha+1}-[a_0,\cdots,a_m,0]t^{\alpha+1}        \right)\\
  &=[a_0,\cdots,a_m,c,0]t^{\alpha+1}.  \\
  \end{split}
  \]
  Finally, the Leibniz formula with the first node $0$ gives
  \[
  [0,a_0,\ldots,a_m,c](t\,t^\alpha)
       =[a_0,\ldots,a_m,c]t^\alpha,
  \]
  because the term multiplied by the first node is zero. Since
  $\alpha>0$, $t^\alpha$ extends continuously to $0$, so this step and
  the preceding endpoint evaluation are valid. Coincident positive nodes
  follow by continuity.
   $\hfill\Box$\\  
  \begin{lem}\label{l1}
   Let $s>1$ and $(k_1,\cdots,k_r)\in(\mathbb Z^+)^r$.
   Put $K_r=k_1+\cdots+k_r$ and define 
  \[
  X_j=x_1\cdots x_{k_1+\cdots +k_j}, \quad 1\leq j\leq r
  \]
  and 
  \[
  R_{k_1,\cdots,k_r}=\int_{(0,1)^{k_1+\cdots+k_r}} [1,X_1,\cdots, X_{r-1},X_r,X_r]t^{s+r}dx_1\cdots dx_{K_r}.
  \]
  Then  \[
R_{k_1,\cdots,k_r} =
\begin{cases}
s, & (k_1,\cdots,k_r)=(\{1\}^r),\\
 H_s^\star(k_1,\cdots ,k_{i-1},k_i-1) , &k_i\geq 2,k_{i+1}=\cdots=k_r=1.
\end{cases}
\]  \end{lem}
  \noindent{\bf Proof:}
    By Lemma \ref{eqp}, we have
    \[
    \begin{split}
    &\;\;\;\; R_{\{1\}^r}      \\
    &=\int_{(0,1)^r} [1,X_1,\cdots,X_{r-1}, X_r,X_r]t^{s+r}dx_1\cdots dx_r             \\
    &=     \int_{(0,1)^{r-1}} [1,X_1,\cdots,X_{r-2}, X_{r-1},X_{r-1}]t^{s+r-1}dx_1\cdots dx_{r-1}  \\  
    &=\int^1_0[1,x_1,x_1]t^{s+1} dx_1        \\
    &= [1,1]t^s\\
    &=s. 
    \end{split}
    \]
     For $(k_1,\cdots,k_r)\neq (\{1\}^r) $ satisfying
     \[
     k_i\geq 2,k_{i+1}=\cdots=k_r=1,     \]
     \begin{samepage}
     repeated application of Lemma~\ref{eqp} first eliminates the
     trailing unit blocks. In the last remaining block put
     $c=x_1\cdots x_{K_i-1}$, where $K_i=k_1+\cdots+k_i$,
     and integrate in $x_{K_i}$. The resulting integrand is
     \[
       [1,X_1,\ldots,X_{i-1},c]t^{s+i-1}.
     \]
     Because $k_i\ge2$, its remaining variables form the blocks
     $(k_1,\ldots,k_{i-1},k_i-1)$. The cube formula therefore gives
     \[
     R_{k_1,\cdots,k_r} = H_s^\star(k_1,\cdots ,k_{i-1},k_i-1).    \]
   \nopagebreak[4]
   $\hfill\Box$\\
     \end{samepage}
    \begin{lem}\label{l2}
   Let $s>1$. For $(k_1,\cdots,k_r)\in \left(\mathbb{Z}^+\right)^r$, 
   \[
    m\left(     Z_{k_1,\cdots,k_r}(s)  \right)\leq m\left(     Z_{\{1\}^r}(s)  \right),            \]
   \end{lem}
  \noindent{\bf Proof:}
Write $K_j=k_1+\cdots+k_j$, $K_0=0$, and
$X_j=x_1\cdots x_{K_j}$. Define
\[
 \mathcal M_s(X_1,\ldots,X_r)
 =X_r[1,X_1,\ldots,X_{r-1},X_r,X_r]t^{s+r-1}.
\]
The linear Leibniz identity \eqref{eq:dd-linear-product}
identifies this as the difference of the integrands in
Lemma~\ref{l1} and Theorem~\ref{noint}. Hence
\begin{equation}\label{eq:general-interval-length}
 m\left(Z_{k_1,\ldots,k_r}(s)\right)
 =\int_{(0,1)^{K_r}}\mathcal M_s(X_1,\ldots,X_r)\,dx,
\end{equation}
and, in particular,
\begin{equation}\label{eq:ones-interval-length}
 m\left(Z_{\{1\}^r}(s)\right)
 =\int_{(0,1)^r}\mathcal M_s(x_1,x_1x_2,\ldots,x_1\cdots x_r)\,dx.
\end{equation}
Put $Y_j=y_1\cdots y_j$ and
$\Phi_s(y_1,\ldots,y_r)=\mathcal M_s(Y_1,\ldots,Y_r)$.
Hermite--Genocchi gives
\[
 \Phi_s(y)=\frac{\Gamma(s+r)}{\Gamma(s-1)}
          \int_{S_{r+1}}Y_r\mathcal L^{s-2}\,du,
 \qquad
 \mathcal L=u_0+\sum_{i=1}^{r-1}u_iY_i
                 +(u_r+u_{r+1})Y_r.
\]
For each $j$, put $\mathcal P_j=y_j\partial_{y_j}\mathcal L$.
Since $0\le\mathcal P_j\le\mathcal L$,
\[
 \frac{y_j\partial_{y_j}(Y_r\mathcal L^{s-2})}
      {Y_r\mathcal L^{s-2}}
 =1+(s-2)\frac{\mathcal P_j}{\mathcal L}
 \ge\min\{1,s-1\}>0.
\]
Thus $\Phi_s$ is increasing in every coordinate. Since
$\prod_{i=K_{j-1}+1}^{K_j}x_i\le x_{K_j}$, we have
\[
 \mathcal M_s(X_1,\ldots,X_r)
 =\Phi_s\left(
   \prod_{i=1}^{K_1}x_i,\ldots,
   \prod_{i=K_{r-1}+1}^{K_r}x_i\right)
 \le\Phi_s(x_{K_1},\ldots,x_{K_r}).
\]
Integrating, and using the unit volume of the unused variables,
\eqref{eq:general-interval-length} and
\eqref{eq:ones-interval-length} give the assertion.
$\hfill\Box$\\

      For $s>1$ and  $x\in (1,s]$, if $x=s$, by Lemma \ref{rlim}, one has 
    \[
    x=s=\mathfrak{h}_s(({\{1\}^{\infty}}))   .
     \]
     If $x=H_s^\star(k_1,\cdots,k_r)$, by Lemma \ref{rlim}, one has 
     \[
     x=H_s^\star(k_1,\cdots,k_r)  =\mathfrak{h}_s\left( (k_1,\cdots, k_{r-1}, k_r+1,\{1\}^{\infty}) \right) .  \]
     If $$x\neq H_s^\star(k_1,\cdots,k_r)$$ for all $(k_1,\cdots,k_r)\in \left( \mathbb{Z}^+  \right)^r $, $r\geq 1$, from the basic properties of $$Z_{k_1,\cdots,k_r}(s),$$ there is a 
     \[
     {\bf k}=(k_1,\cdots,k_r,\cdots)\in \widehat{\mathcal{T}}
     \]
     such that
     \[
     x\in Z_{k_1,\cdots,k_r}(s)
     \]
     for all $r\geq 1$.
     As a result,
     \[
              H_s^\star (k_1,\cdots,k_r)  <x\leq  H_s^\star (k_1,\cdots,k_r) +   m\left(     Z_{k_1,\cdots,k_r}(s)  \right).   \]
              
              Since by Lemma \ref{l2}
              \[
              \begin{split}
              &\;\;\;\; m\left(     Z_{k_1,\cdots,k_r}(s)  \right)\leq m\left(     Z_{\{1\}^r}(s)  \right),                         \end{split}              
              \]
              we have 
              \[
              \mathop{\mathrm{lim}}_{r\rightarrow +\infty} m\left(     Z_{k_1,\cdots,k_r}(s)  \right) =   \mathop{\mathrm{lim}}_{r\rightarrow +\infty} \left(  s-H_s^\star\left(\{1\}^r\right)   \right) =0.           \]
              Therefore,
              \[
              x= \mathop{\mathrm{lim}}_{r\rightarrow +\infty} H_s^\star (k_1,\cdots,k_r) =\mathfrak{h}_s({\bf k}).                            \]
  In a word,  the map $\mathfrak{h}_s: \widehat{\mathcal{T}}\rightarrow  (1,s]   $ is bijective for $s>1$.  Theorem \ref{com} $(ii)$ is proved.
  
  \begin{lem}\label{bgx}
Let $0<X<1$ and 
  \[
  0<y_r<\cdots<y_2<y_1=1.
  \]
  If $\mathrm{Re}(s)>0$, then 
  \[
  \begin{split}
 &\;\;\;\; [1,Xy_1,Xy_2,\cdots, Xy_r]t^{s+r-1}\\
 &=\sum_{n=1}^{+\infty} \left( X^{n-1}[y_1,\cdots,y_r]t^{n+r-2}      -X^{s+n-1}[y_1,\cdots,y_r]t^{s+n+r-2}\right).\\
 \end{split}
  \]
  \end{lem}
  \noindent{\bf Proof:}
We will need the following formula:\[
 [1,Xy_1,\ldots,Xy_r]t^{s+r-1}
 =[y_1,\ldots,y_r]
   \left(t^{r-1}\frac{1-X^st^s}{1-Xt}\right).
\]
Recall the standard identity for divided differences (which can be deduced from Leibnitz formula):
\[
[a,z_1,\ldots,z_r]f
=
[z_1,\ldots,z_r]
\left(\frac{f(t)-f(a)}{t-a}\right).
\]
Applying this identity with $a=1$ and
$f(t)=t^{s+r-1}$, we obtain
\begin{align*}
&[1,Xy_1,\ldots,Xy_r]t^{s+r-1}\\
&\quad=
[Xy_1,\ldots,Xy_r]
\left(\frac{1-t^{s+r-1}}{1-t}\right)\\
&\quad=
\sum_{k=0}^{+\infty}
[Xy_1,\ldots,Xy_r]\left(t^k-  t^{k+s+r-1} \right)\\
&\quad=  \sum_{k=0}^{+\infty}
X^{k-r+1}[y_1,\ldots,y_r] t^k-  \sum_{k=0}^{+\infty}  X^{k+s} [y_1,\cdots,y_r]  t^{k+s+r-1}       \\
&\quad=  \sum_{k=r-1}^{+\infty}
X^{k-r+1}[y_1,\ldots,y_r] t^k-  \sum_{k=0}^{+\infty}  X^{k+s} [y_1,\cdots,y_r]  t^{k+s+r-1}       \\
&\quad=  [y_1,\ldots,y_r]  \left[t^{r-1}\cdot \left(\sum_{k=r-1}^{+\infty}  (Xt)^{k-r+1}-   \sum_{k=0}^{+\infty}   (Xt)^{k+s}  \right)    \right]    \\
&\quad=
[y_1,\ldots,y_r]
\left(t^{r-1}\frac{1-X^st^s}{1-Xt}\right).
\end{align*}
Here, the third equality follows from the homogeneity
property
\[
[Xy_1,\ldots,Xy_r]t^k
=
X^{k-r+1}[y_1,\ldots,y_r]t^k,
\qquad k\ge 0.
\]
The fourth equality follows from the fact that
divided differences of order $r-1$ vanish on
polynomials of degree less than $r-1$.
This completes the proof.
Expand $(1-Xt)^{-1}$ into its absolutely uniformly convergent
geometric series on $[y_r,1]$. Applying the finite divided
difference termwise gives the asserted identity.
$\hfill\Box$\\
         
         \begin{lem}\label{spn}         For $\mathrm{Re}(s)>0$ and $(k_1,\cdots,k_r)\in \left(\mathbb{Z}^+\right)^r$, one has 
         \[
         \bigg{|}   H_n^\star(k_1,\cdots,k_r)-   H_{s+n}^\star(k_1,\cdots,k_r) \bigg{|}\leq \frac{c_{s,{\bf k}} \left( \mathrm{log}\,(n+2)   \right)^{r-1}}{n+1},\qquad n\geq 1.
         \]
         Here $c_{s,{\bf k}}$ depends only on $s$ and ${\bf k}=(k_1,\cdots,k_r)$.
         \end{lem}
         \noindent{\bf Proof:}
         The integral representation and Lemma \ref{ker}, $(iii)$, give
         \[
         \begin{split}
         H_n^\star({\bf k})-H_{n+s}^\star({\bf k})
          &=\int_{(0,1)^{K_r}}
             [1,X_1,\ldots,X_r]
                 (t^{n+r-1}-t^{n+s+r-1})\,dx\\
          &=\int_{(0,1)^{K_r}}(P_{n,r}(s;X)-A_r(s;X))\,dx.
         \end{split}
         \]
         The assertion is exactly the estimate in Lemma \ref{conv}
         with $M=n$. $\hfill\Box$\\

         To justify termwise integration in Lemma \ref{bgx}, write
         \[
         X=X_1,\qquad Y=(X_2/X_1,\ldots,X_r/X_1).
         \]
         Its $n$th summand is
         \[
         X^{n-1}(1-X^s)A_{r-1}(n;Y)
          +X^{n+s-1}\bigl(A_{r-1}(n;Y)-A_{r-1}(n+s;Y)\bigr).
         \]
         The integrated absolute value of the first term is at most
         \[
         C_s\bigl(n^{-k_1}-(n+1)^{-k_1}\bigr)
                  H_n^\star(k_2,\ldots,k_r),
         \]
         by \eqref{eq:power-difference-bound}. The second term vanishes
         when $r=1$; for $r\ge2$, Lemma~\ref{conv} applied to the tail bounds it by
         \[
         \frac{C_s\log^{r-2}(n+2)}{(n+\Rea s)^{k_1}(n+1)}.
         \]
         Both bounds are summable, since
         $H_n^\star(k_2,\ldots,k_r)\leq(1+\log n)^{r-1}$.
         Here the empty-tail value is $1$. Consequently the summands are
         absolutely summable in $L^1$ on the cube. Fubini's theorem gives
         the following identity, with an absolutely convergent scalar series: 
         \begin{equation}\label{eq:nested-series}
         \begin{split}
         &\;\;\;\; H_s^\star(k_1,\cdots,k_r)         \\
         &=  \int_{(0,1)^{K_r}}  [1,X_1,\cdots,X_r] t^{s+r-1}   dx_1\cdots dx_{K_r}              \\
         &=\int_{(0,1)^{K_r}}\sum_{n=1}^{\infty}
              \bigl(X^{n-1}A_{r-1}(n;Y)
                 -X^{s+n-1}A_{r-1}(s+n;Y)\bigr)\,dx\\
         &= \sum_{n=1}^{+\infty} \left(\frac{H_n^\star(k_2,\cdots,k_r)}{n^{k_1}} -  \frac{H_{s+n}^\star(k_2,\cdots,k_r)}{(s+n)^{k_1}}    \right) . \\
              \end{split}
 \end{equation}
 To prove Theorem~\ref{com}, $(iii)$, we need estimates that are
 uniform in the depth of finite prefixes of each fixed infinite index.
 The next two lemmas provide these estimates. At the end of this
 section we pass to the infinite-depth limit and prove the precise
 assertion, with $|\Im s|\leq T$ for each fixed $T$.
         \begin{lem}\label{zexp}
         For every nonempty positive integer index ${\bf k}$ there is
         a positive measurable function $f_{\bf k}$ on $(0,\infty)$ such that
         \[
         \int_0^\infty\min\{1,u\}f_{\bf k}(u)\,du<\infty,
         \qquad
         H_s^\star({\bf k})=\int_0^\infty(1-e^{-su})f_{\bf k}(u)\,du
         \quad(\Rea s>0).
         \]
         \end{lem}
         \noindent{\bf Proof:}
At depth one, the absolutely convergent difference series gives
\begin{equation}
 H_s^\star(k)=\sum_{n=1}^\infty
 \left(n^{-k}-(n+s)^{-k}\right),\label{eq:depth-one-series-original}
\end{equation}
and hence
\[
 f_k(u)=\frac{u^{k-1}}{(k-1)!(e^u-1)}.
\]
Its required integrability follows directly at zero and infinity.
Suppose the result holds for the nonempty tail
$\boldsymbol\ell=(k_2,\ldots,k_r)$, and put $k=k_1$.
Define
\[
 g_1(x)=\int_x^\infty f_{\boldsymbol\ell}(u)\,du,\qquad
 g_{j+1}(x)=\int_0^xg_j(v)\,dv.
\]
The induction hypothesis makes $g_1(x)$ finite for $x>0$, and
Tonelli gives, for every $R>0$,
\[
 \int_0^R g_1(x)\,dx
 =\int_0^\infty\min\{R,u\}f_{\boldsymbol\ell}(u)\,du<\infty.
\]
Thus every $g_j(x)$ is finite and strictly positive for $x>0$.
For real $z>0$, Tonelli and induction on $j$ give
\[
 \int_0^\infty e^{-zx}g_j(x)\,dx
 =\frac{H_z^\star(\boldsymbol\ell)}{z^j}.
\]
For real $s>0$, the already justified series
\eqref{eq:nested-series} and Tonelli therefore yield
\[
 \begin{split}
 H_s^\star(k,\boldsymbol\ell)
 &=\sum_{n\ge1}\int_0^\infty
     e^{-nx}(1-e^{-sx})g_k(x)\,dx\\
 &=\int_0^\infty(1-e^{-sx})\frac{g_k(x)}{e^x-1}\,dx.
 \end{split}
\]
Set $f_{k,\boldsymbol\ell}(x)=g_k(x)/(e^x-1)>0$.
Taking $s=1$ gives
\[
 \int_0^\infty(1-e^{-x})f_{k,\boldsymbol\ell}(x)\,dx
 =H_1^\star(k,\boldsymbol\ell)=1.
\]
Since $1-e^{-x}$ is comparable to $\min\{1,x\}$, this proves
integrability. On compact subsets of $\Rea s>0$,
$|1-e^{-sx}|\le C\min\{1,x\}$, so the integral is holomorphic
there. The identity theorem extends the formula to all $\Rea s>0$,
completing the induction.
$\hfill\Box$\\

       \begin{rem}\label{nze}
       By Lemma \ref{zexp}, as $|e^{-su}|<1$ for $\mathrm{Re}(s)>0$ and $u>0$, it is clear that 
        \[
         H_s^\star (k_1,\cdots,k_r)\neq 0    \]
       for $\mathrm{Re}(s)>0$. Furthermore,
        \[
       \mathrm{Re}\left(  H_s^\star (k_1,\cdots,k_r)\right)>0   \]
       for $\mathrm{Re}(s)>0$.                   \end{rem}
                   
         \begin{lem}\label{zlog} If $r,m\in \mathbb{Z}^+$ and $(k_1,\cdots,k_r) \; \rotatebox[origin=c]{180}{\text{ $\succ$}}  (\{1\}^{m})$,        
         then         \[
         \Big{|}  H_s^\star(k_1,\cdots,k_r)     \Big{|}\leq 3\left(1+ \mathrm{log}\,(1+|s|) \right)^m
         \]
         for $\Rea s>1$. Moreover, for every finite positive integer index,
         independently of the preceding order restriction,
         \[
          \Big{|}  H_s^\star(k_1,\ldots,k_r)     \Big{|}\leq  3|s|      \]
         for $\mathrm{Re}(s)> 1$ and $r\geq 1$.
         \end{lem}
            \noindent{\bf Proof:}
For real $a>1$, put $n=[a]+1$. Under the stated order restriction,
Theorem~\ref{rord} and Lemma~\ref{zexp} give
\[
 \begin{split}
 H_a^\star(k_1,\ldots,k_r)
 &\le H_a^\star(\{1\}^m)
 <H_n^\star(\{1\}^m)
 \le\bigl(H_n^\star(1)\bigr)^m\\
 &\le(1+\log n)^m
 \le\bigl(1+\log(1+a)\bigr)^m.
 \end{split}
\]
The elementary bound
$|1-e^{-w}|\le\min\{|w|,2\}$ implies
\begin{equation}
 |1-e^{-w}|\le3(1-e^{-|w|}),\qquad
 \Rea w>0.\label{eq:complex-exponential-bound}
\end{equation}
Hence Lemma~\ref{zexp} yields, for $\Rea s>1$,
\[
 |H_s^\star(k_1,\ldots,k_r)|
 \le3H_{|s|}^\star(k_1,\ldots,k_r),
\]
which proves the first bound.
For every index and real $a\ge1$, the inequality
$1-e^{-au}\le a(1-e^{-u})$ and the normalization
$H_1^\star(k_1,\ldots,k_r)=1$ give
\[
 H_a^\star(k_1,\ldots,k_r)
 \le aH_1^\star(k_1,\ldots,k_r)=a.
\]
Taking $a=|s|$ in the preceding complex estimate proves
the second bound.
$\hfill\Box$\\
       
  We finish the proof of Theorem \ref{com}, $(iii)$, with the quantifiers
  made explicit. For every fixed ${\bf k}\in\mathcal T$ and $T<\infty$,
  \[
  \lim_{\sigma\to\infty}\ 
      \sup_{|\tau|\leq T}
      |\mathfrak h_{\sigma+i\tau}({\bf k})-\eta({\bf k})|=0.
  \]
  The assertion is pointwise in the infinite index.  Such uniform convergence would be impossible: for each real $\sigma>1$,
  all values $\mathfrak h_\sigma({\bf k})$ are at most $\sigma$,
  whereas $\eta(\mathcal T)=(1,\infty)$ is unbounded.

  For a finite index with $k_1\geq2$, the nested series proved above gives
  \[
  H_s^\star(k_1,\ldots,k_r)-\zeta^\star(k_1,\ldots,k_r)
    =-\sum_{n=1}^\infty
        \frac{H_{s+n}^\star(k_2,\ldots,k_r)}{(s+n)^{k_1}}.
  \]
  The empty-tail value is understood to be $1$ and satisfies each
  majorant below directly. For a nonempty tail we use Lemma~\ref{zlog}.
  If $(k_1,\ldots,k_r)\prec(2,\{1\}^{m})$ and $\sigma=\Rea s>1$,
  Lemma \ref{zlog} gives the uniform-in-depth bound
  \[
  \left|H_s^\star(k_1,\ldots,k_r)-\zeta^\star(k_1,\ldots,k_r)\right|
  \leq\begin{cases}
  \displaystyle\sum_{n\geq1}
       \frac{3(1+\log(1+|s+n|))^m}{(\sigma+n)^2},&k_1=2,\\[6pt]
  \displaystyle\sum_{n\geq1}
       \frac{3|s+n|}{(\sigma+n)^3},&k_1\geq3.
  \end{cases}
  \]
  Here $\prec$ denotes the reverse of $\succ$. The second line uses the
  linear bound for every finite tail, regardless of its depth.
  For $|\Im s|\leq T$, the two right-hand sides are respectively
  \[
  O_{m,T}\!\left(\frac{(1+\log\sigma)^m}{\sigma}\right),
  \qquad O_T(\sigma^{-1}).
  \]
  Indeed, $|s+n|\leq\sigma+n+T$, and comparison with the integrals of
  $(C_T+\log x)^m/x^2$ and $x^{-2}$ proves these estimates.

  If ${\bf k}\in\mathcal T$ has $k_1=2$, choose $j\geq2$ with
  $k_j\geq2$ and put $m=j-1$. Every sufficiently long finite prefix then
  lies below $(2,\{1\}^{m})$. If $k_1\geq3$, the second bound applies
  to every prefix. In either case the bound is independent of prefix
  length, so we may pass to the limits defining $\mathfrak h_s({\bf k})$
  and $\eta({\bf k})$, and then let $\sigma\to+\infty$. This proves the
  asserted convergence uniformly for $|\Im s|\leq T$ for the fixed index.
  Theorem \ref{com}, $(iii)$, is proved.\\

\section{Separation by two positive half-planes}\label{sec:complex-injectivity}

We prove Theorem~\ref{thm:complex-injectivity} by allowing two
ways to separate a pair of indices. For its difference $D(s)$,
we shall establish either $\Rea D(s)>0$ or $\Rea(D(s)/s)>0$
throughout $\Rea s>1$. The choice is fixed for the pair and does
not depend on $s$. These alternatives have explicit transition
rules when a common entry is prepended. The analytic ingredient
is the following summation principle.

\begin{lem}\label{lem:new-summation}
Let $g$ be holomorphic on $\Rea s>0$, real on the positive real axis,
and satisfy $g(1)=0$ and $\Rea g(s)>0$ for $\Rea s>1$.
Suppose that, uniformly on $\Rea s\geq1$,
$g(s)=O(\log^M(2+|s|))$ for some $M\geq0$.
For an integer $p\geq0$, put
\[
 (\mathcal D_p g)(s)=\sum_{n=1}^{\infty}
 \left(\frac{g(n)}{n^p}-\frac{g(n+s)}{(n+s)^p}\right).
\]
For $p=0$, assume additionally
$g(s)=O(\log^M(2+|s|)/|s|)$ on this half-plane.
The series of paired differences converges absolutely, locally
uniformly for $\Rea s>1$, and
\[
 \begin{array}{ll}
 \Rea\bigl((\mathcal D_p g)(s)/s\bigr)>0,&p=0,1,2,\\[2pt]
 \Rea(\mathcal D_p g)(s)>0,&p\geq3.
 \end{array}
\]
\end{lem}
Its proof is given in Section~\ref{sec:new-summation-proof}.
The pairing in the definition is retained throughout; convergence
of the two separate series is not required.

We first record two consequences of Lemma~\ref{zexp}.
For a nonempty finite index $\mathbf v$, normalization at $s=1$ gives
\[
 H_s^\star(\mathbf v)-1
 =\int_0^\infty e^{-u}(1-e^{-(s-1)u})f_{\mathbf v}(u)\dd u,
 \qquad \Rea s>1.
\]
Its real part is strictly positive, since for $\sigma=\Rea s>1$
\[
 \Rea(1-e^{-(s-1)u})\geq1-e^{-(\sigma-1)u}>0
 \qquad(u>0).
\]
Also, if $q_{\mathbf v}(x)=\int_x^\infty f_{\mathbf v}(u)\dd u$, then
\[
 \frac{H_s^\star(\mathbf v)}s
 =\int_0^\infty e^{-sx}q_{\mathbf v}(x)\dd x,
 \qquad \int_0^\infty e^{-x}q_{\mathbf v}(x)\dd x=1.
\]
Here $q_{\mathbf v}(x)$ is finite and strictly positive for $x>0$.
Tonelli at $s=1$ and absolute Fubini justify these identities.
Consequently
\begin{equation}\label{eq:new-unit-bound}
 \left|\frac{H_s^\star(\mathbf v)}s\right|
 \leq\int_0^\infty e^{-\sigma x}q_{\mathbf v}(x)\dd x
 <\int_0^\infty e^{-x}q_{\mathbf v}(x)\dd x=1
 \qquad(\sigma>1).
\end{equation}
The same strict bound holds for the empty index, whose value is $1$,
because $|s|>1$.

\begin{Thm}\label{thm:halfplane-order}
For any two finite indices $\mathbf k\succ\mathbf l$, with the empty
index admitted as the least element, there is
$\epsilon\in\{0,1\}$ such that
\begin{equation}\label{eq:halfplane-order}
 \Rea\frac{H_s^\star(\mathbf k)-H_s^\star(\mathbf l)}{s^\epsilon}>0
 \qquad(\Rea s>1).
\end{equation}
The choice of $\epsilon$ depends only on the common prefix and the
first unequal entries, or on the common prefix in the prefix case.
\end{Thm}
\noindent{\bf Proof:}
Call a difference $D$ of type $A$ if $\Rea D>0$, and of type $B$
if $\Rea(D/s)>0$, on $\Rea s>1$.
These labels need not be exclusive; the following rules specify
which label we choose at each step.
Prepending a common entry $a$ replaces $D$ by
$\mathcal D_aD$, by \eqref{eq:nested-series}.
Lemma~\ref{lem:new-summation}, applied to $g=D$ in type $A$ and to
$g=D/s$ in type $B$, gives the rules
\[
 \begin{array}{c|ccc}
  &a=1,2&a=3&a\geq4\\ \hline
 A&B&A&A\\
 B&B&B&A
 \end{array}
\]
Indeed, in the second case the exact identity is
$\mathcal D_aD=\mathcal D_{a-1}(D/s)$.
Every function to which the lemma is applied has the required
growth: for an index of depth $r$, choose an all-ones index of
length greater than $r$ in Lemma~\ref{zlog}.
This gives a logarithmic power bound for each finite value and
each finite difference. The bounds extend to $\Rea s=1$ by
continuity. Division by $s$ gives the additional decay needed
when $a=1$ in type $B$, so that $p=a-1=0$.
All these functions are holomorphic on $\Rea s>0$, vanish at
$s=1$, and are real on the positive real axis.

It remains to start the induction at the first difference.
Write the two branches as $(a,\mathbf v)$ and $(b,\mathbf w)$,
where $a<b$, and put $d=b-a$ and
\[
 R(s)=H_s^\star(\mathbf v)-s^{-d}H_s^\star(\mathbf w).
\]
This function has type $A$. In fact, including empty tails,
\[
 \begin{split}
 \Rea R(s)
 &\geq\Rea(H_s^\star(\mathbf v)-1)
       +1-|s|^{-(d-1)}\left|H_s^\star(\mathbf w)/s\right|\\
 &>0
 \qquad(\Rea s>1),
 \end{split}
\]
by \eqref{eq:new-unit-bound} and $d\geq1$.
Also $R(1)=0$; $R$ is holomorphic on $\Rea s>0$, real on the
positive axis, and has the required logarithmic growth.
The nested series gives exactly
\[
 H_s^\star(a,\mathbf v)-H_s^\star(b,\mathbf w)
       =(\mathcal D_aR)(s).
\]
Choose type $B$ for this branch difference when $a=1,2$, and
type $A$ when $a\geq3$. Then apply the table through the common
prefix, from right to left.
If one index is a prefix of the other, start instead with
$H_s^\star(\mathbf v)-1$, where $\mathbf v$ is the nonempty
remaining tail, and choose type $A$ before applying the same table.
Thus the final choice is determined by the stated prefix data,
independently of all later entries. Taking $\epsilon=0$ for the
chosen type $A$ and $\epsilon=1$ for type $B$ proves the assertion.
Totality of the order proves Theorem~\ref{thm:complex-injectivity}, $(i)$.
$\hfill\Box$\\

\begin{Thm}
\label{thm:infinite-halfplane-order}
If $\mathbf k,\mathbf l\in\widehat{\mathcal T}$ and
$\mathbf k\succ\mathbf l$ in reversed lexicographic order, then, for
some $\epsilon\in\{0,1\}$ independent of $s$,
\begin{equation}\label{eq:infinite-halfplane-order}
 \Rea\frac{\mathfrak h_s(\mathbf k)-\mathfrak h_s(\mathbf l)}
                 {s^\epsilon}>0\qquad(\Rea s>1).
\end{equation}
\end{Thm}
\noindent{\bf Proof:}
Let $j$ be the first unequal position. For every $r\geq j$, the
paired finite prefixes have the same common prefix and first
unequal entries. The fixed rules in the preceding proof therefore
supply the same $\epsilon$ for all these pairs.
The locally uniform convergence in Theorem~\ref{com}, $(i)$,
shows that their normalized differences converge to a holomorphic
function on $\Rea s>1$. Its real part
\[
 U(s)=\Rea\frac{\mathfrak h_s(\mathbf k)-\mathfrak h_s(\mathbf l)}
                    {s^\epsilon}
\]
is harmonic and nonnegative. Passing to the limit alone gives
only this nonnegativity. The real order proved in Section~4 gives
\[
 U(2)=2^{-\epsilon}
       \bigl(\mathfrak h_2(\mathbf k)-\mathfrak h_2(\mathbf l)\bigr)>0.
\]
The strong minimum principle on the connected half-plane then
gives $U(s)>0$ everywhere there. Totality proves
Theorem~\ref{thm:complex-injectivity}, $(ii)$.
$\hfill\Box$\\

\section{Proof of the summation principle}\label{sec:new-summation-proof}

The proof of Lemma~\ref{lem:new-summation} uses the elementary atoms
\[
 \gamma_t(z)=\frac{z-1}{(z-1)^2+t^2},\qquad t>0.
\]
They have positive real part on $\Rea z>1$, since
\[
 \gamma_t(z)=\frac12\left(\frac1{z-1-it}+\frac1{z-1+it}\right).
\]
They have poles at $1\pm it$, so they are auxiliary kernels and
do not themselves satisfy the holomorphy hypothesis of
Lemma~\ref{lem:new-summation}. We first prove the required
operator inequalities directly for these kernels. A positive-real
integral representation then gives the result for $g$.

\begin{lem}\label{lem:new-atom-bound}
For $x\geq2$, $y\ne0$, $t>0$, and every integer $p\geq3$,
\[
 \Rea\frac{\gamma_t(x+iy)}{(x+iy)^p}
       <\frac{\gamma_t(x)}{x^p}.
\]
\end{lem}
\noindent{\bf Proof:}
By conjugation, it suffices to consider $y>0$. Set
\[
 u=\frac yx,\qquad v=\frac y{x-1},\qquad
 b=\frac{t^2}{(x-1)^2}.
\]
Since $x\geq2$, we have
\[
 0<u\leq v\leq2u.
\]
Also, $\gamma_t(x)/x^p>0$. Thus the required inequality is
equivalent to $\Rea F_p<1$, where
\[
 F_p
 =\frac{\gamma_t(x+iy)/(x+iy)^p}{\gamma_t(x)/x^p}
 =\frac{(1+b)(1+iv)}
        {(1+iu)^p((1+iv)^2+b)}.
\]

\smallskip
\noindent\emph{Step 1:}
Define
\[
 B=\frac{(1+iv)(1-iv/2)}{(1+iu)^p}.
\]
The factor $1-iv/2$ is chosen so that the quotient $B/F_p$
has real part $1$. Indeed,
\[
 \begin{aligned}
 \frac{B}{F_p}
 &=\frac{(1-iv/2)(1+b-v^2+2iv)}{1+b}\\
 &=1+i\,\frac{v(3+v^2-b)}{2(1+b)}.
 \end{aligned}
\]
In particular, $\Rea(B/F_p)=1$. Since $F_p\ne0$,
multiplication by $|F_p|^2$ gives
\[
 \Rea(B\overline{F_p})=|F_p|^2.
\]
Completing the square, we obtain
\[
 \left|F_p-\frac B2\right|^2
 =|F_p|^2-\Rea(B\overline{F_p})+\frac{|B|^2}{4}
 =\frac{|B|^2}{4}.
\]
Thus $F_p$ lies on the circle with centre $B/2$ and radius
$|B|/2$. Its real part therefore satisfies
\[
 \begin{aligned}
 \Rea F_p
 &=\frac{\Rea B}{2}
   +\Rea\left(F_p-\frac B2\right)\\
 &\leq\frac{\Rea B+|B|}{2}.
 \end{aligned}
\]
In particular, $\Rea F_p\leq|B|$.

\smallskip
\noindent\emph{Step 2: }
Using $v\leq2u$, we find
\begin{equation}\label{eq:new-atom-modulus}
 \begin{aligned}
 |B|^2
 &=\frac{(1+v^2)(1+v^2/4)}{(1+u^2)^p}\\
 &\leq\frac{(1+4u^2)(1+u^2)}{(1+u^2)^p}
 =\frac{1+4u^2}{(1+u^2)^{p-1}}.
 \end{aligned}
\end{equation}
If $p\geq5$, then
\[
 (1+u^2)^{p-1}\geq(1+u^2)^4>1+4u^2,
\]
so $|B|<1$. If $u^2>2$, then, for every $p\geq3$,
\[
 (1+u^2)^{p-1}\geq(1+u^2)^2>1+4u^2,
\]
where the strict inequality follows from
\[
 (1+u^2)^2-(1+4u^2)=u^2(u^2-2)>0.
\]
Thus $|B|<1$ in this case as well. In either case,
$\Rea F_p\leq|B|<1$.

\smallskip
\noindent\emph{Step 3: The remaining cases $p=3,4$ and $u^2\leq2$.}
Here we also use the angle of $B$. Put
\[
 \theta=\arctan u,\qquad
 \delta=\arctan v-\arctan(v/2).
\]
Both $\theta$ and $\delta$ belong to $(0,\pi/2)$, and
\[
 \tan\delta
 =\frac{v}{2+v^2}
 <\frac v2
 \leq u
 =\tan\theta.
\]
Since tangent is increasing on $(0,\pi/2)$, we obtain
\[
 0<\delta<\theta
 \leq\arctan\sqrt2<\frac{\pi}{3}.
\]

The definition of $B$ now gives
\[
 B=|B|e^{-i\alpha},
 \qquad
 \alpha=p\theta-\delta.
\]
Since $p=3,4$ and $0<\delta<\theta$, we have
\[
 2\theta<\alpha<4\theta.
\]
Moreover, $\theta<\pi/3$ implies
$4\theta<2\pi-2\theta$. Consequently,
\[
 \theta<\frac{\alpha}{2}<\pi-\theta.
\]
Cosine is decreasing on $[0,\pi]$, so
\[
 -\cos\theta
 <\cos(\alpha/2)
 <\cos\theta.
\]
Because $\cos\theta>0$, this yields
\[
 \cos^2(\alpha/2)
 \leq\cos^2\theta
 =\frac{1}{1+u^2}.
\]

Finally, the circle estimate from Step~1 and the identity
$\Rea B=|B|\cos\alpha$ give
\[
 \begin{aligned}
 \Rea F_p
 &\leq\frac{|B|+\Rea B}{2}
   =|B|\frac{1+\cos\alpha}{2}\\
 &=|B|\cos^2(\alpha/2)\\
 &\leq\frac{|B|}{1+u^2}\\
 &\leq
   \frac{\sqrt{1+4u^2}}{(1+u^2)^{(p+1)/2}}\\
 &\leq
   \frac{\sqrt{1+4u^2}}{(1+u^2)^2}
 <1.
 \end{aligned}
\]
Here the penultimate inequality uses $p\geq3$, and the last
inequality follows from $(1+u^2)^4>1+4u^2$ for $u>0$.
Thus $\Rea F_p<1$ in every case, which proves the lemma.
$\hfill\Box$\\

For the three smaller powers, we use one cosine-transform observation.
For $x>0$ and $y>0$, put
\[
 K_y(x)=\frac{1-\cos(yx)+y\sin(yx)}{e^x-1},\qquad
 (Ef)(x)=\int_0^x e^{-(x-v)}f(v)\dd v.
\]

\begin{lem}\label{lem:new-small-powers}
For $y,t>0$ and $p=0,1,2$,
\[
 \int_0^{+\infty} K_y(x)(E^p\cos(t\,\cdot))(x)\dd x>0.
\]
\end{lem}
\noindent{\bf Proof:}
Extend $K_y$ evenly and continuously to the real line, with
$K_y(0)=y^2$. The kernel itself need not be nonnegative.
We shall prove that its Fourier density $k_y$ is positive, even,
and nonincreasing on $(0,+\infty)$.
Here the Fourier convention is
$K_y(x)=\int_{\mathbb R}e^{iux}k_y(u)\dd u$.
To verify the fact, put $v=y/2$ and $a=1/2$ and factor
\[
 K_y(x)=\frac{x}{2\sinh(x/2)}\cdot
        \frac{2\sin(vx)}x\cdot
        e^{-|x|/2}\bigl(\sin(v|x|)+2v\cos(vx)\bigr).
\]
The first factor has the positive, even, decreasing density
$\rho(u)=\pi/(2\cosh^2(\pi u))$.
Indeed, the substitution $r=e^{2\pi u}$ gives
\[
 \begin{split}
 \int_{\mathbb R}e^{iux}\rho(u)\dd u
 &=\int_0^{+\infty}\frac{r^{ix/(2\pi)}}{(1+r)^2}\dd r\\
 &=\Gamma(1+ix/(2\pi))\Gamma(1-ix/(2\pi))
 =\frac{x}{2\sinh(x/2)}.
 \end{split}
\]
The value at $x=0$ is understood by continuity.
The second factor has density $\mathbf1_{[-v,v]}$.
The last factor has density
\[
 h(u)=\frac{2v(a^2+v^2)}
 {\pi\bigl(a^2+(u+v)^2\bigr)\bigl(a^2+(u-v)^2\bigr)}.
\]
Although $h$ need not decrease, its convolution with
$\mathbf1_{[-v,v]}$ does: its derivative is
$$h(u+v)-h(u-v)<0$$ for $u>0$, since the difference of the corresponding
denominators is $8\pi uv(u^2+a^2)>0$.
Convolution with the positive, even, decreasing density $\rho$
preserves evenness and nonincrease. To see the latter assertion,
write an even nonincreasing nonnegative function as a positive
mixture of interval indicators; for each $c>0$,
\[
 \frac{d}{du}(\rho*\mathbf1_{[-c,c]})(u)
 =\rho(u+c)-\rho(u-c)\leq0\qquad(u>0).
\]
Thus $k_y=\rho*\mathbf1_{[-v,v]}*h$ has the asserted properties,
and it is strictly positive everywhere. It is smooth and even.
The rational tails of $h$ and the exponential tails of $\rho$
give $k_y(u)=O(u^{-4})$ and $k_y'(u)=O(u^{-5})$ at infinity.
Integration by parts in
$$K_y(x)=2\int_0^\infty\cos(bx)k_y(b)\dd b$$ gives
\begin{equation}\label{eq:new-sinc-mixture}
 K_y(x)=\int_0^\infty \frac{\sin(bx)}x\,w_y(b)\dd b,
 \qquad w_y(b)=-2k_y'(b)\geq0,
\end{equation}
where $w_y$ is not identically zero: otherwise $k_y$ would be
constant, contrary to positivity and decay. Smooth evenness gives
$w_y(b)=O(b)$ at zero, and the preceding tail estimate gives
\[
 \int_0^{+\infty}(1+b)w_y(b)\dd b<\infty.
\]

We compute the integrals against each function $\sin(bx)/x$.
For $b,t>0$ with $b\ne t$, set
\[
 C=\frac\pi2\mathbf1_{\{t<b\}},\qquad
 J=\frac12\log\frac{b+t}{|b-t|},\qquad
 L=\arctan b,\qquad M=\frac b{1+b^2}.
\]
The needed elementary integrals, interpreted by Abel limits if
necessary, are
\[
 \begin{aligned}
 \int_0^\infty\frac{\sin(bx)\cos(tx)}x\dd x&=C,&
 \int_0^\infty\frac{\sin(bx)\sin(tx)}x\dd x&=J,\\
 \int_0^\infty e^{-x}\frac{\sin(bx)}x\dd x&=L,&
 \int_0^\infty e^{-x}\sin(bx)\dd x&=M.
 \end{aligned}
\]
They give
\[
\begin{split}
 I_0&=C,\\
 I_1&=\frac{C-L+tJ}{1+t^2},\\
 I_2&=\frac{(1-t^2)(C-L)+2tJ-(1+t^2)M}{(1+t^2)^2},
\end{split}
\qquad
 I_p=\int_0^\infty\frac{\sin(bx)}x(E^p\cos(t\,\cdot))(x)\dd x.
\]
Indeed,
\[
 E^p e^{itx}
 =\frac{e^{itx}-e^{-x}\sum_{j=0}^{p-1}(1+it)^jx^j/j!}
        {(1+it)^p},\qquad p=1,2.
\]

We now prove that $I_1$ and $I_2$ are strictly positive for all
$b,t>0$ with $b\ne t$. Since their denominators are positive,
it suffices to prove the positivity of their numerators
\[
 \begin{aligned}
 N_1&=C-L+tJ,\\
 N_2&=(1-t^2)(C-L)+2tJ-(1+t^2)M.
 \end{aligned}
\]
We first record two elementary estimates. For every $x>0$,
\[
 \frac{x}{1+x^2}
 <
 \int_0^x\frac{\dd u}{1+u^2}
 =\arctan x
 <x.
\]
Indeed, for $0<u<x$ the integrand lies strictly between
$1/(1+x^2)$ and $1$. Similarly, for every $0<r<1$,
\[
 \frac12\log\frac{1+r}{1-r}
 =\int_0^r\frac{\dd u}{1-u^2}
 >r,
\]
because the integrand is strictly greater than $1$ on $(0,r)$.\\
\smallskip
\noindent\emph{Case 1: $t>b$.}
Here $C=0$. Applying the first estimate with $x=b$ gives
\[
 M=\frac{b}{1+b^2}<L=\arctan b<b.
\]
Also, since $0<b/t<1$, the second estimate gives
\[
 J=\frac12\log\frac{1+b/t}{1-b/t}>\frac bt,
 \qquad\text{hence}\qquad tJ>b.
\]
Consequently,
\[
 N_1=tJ-L>b-L>0.
\]
For the second numerator, we obtain
\[
 \begin{aligned}
 N_2
 &=(t^2-1)L+2tJ-(1+t^2)M\\
 &=t^2(L-M)+2tJ-L-M\\
 &>t^2(L-M)+2b-L-M\\
 &=t^2(L-M)+(b-L)+(b-M)>0.
 \end{aligned}
\]
The last inequality follows from $M<L<b$, which makes each
term in the last expression strictly positive.

\smallskip
\noindent\emph{Case 2: $0<t<b$.}
Here $C=\pi/2$. Put
\[
 Q=C-L=\frac{\pi}{2}-\arctan b=\arctan(1/b).
\]
Applying the first elementary estimate with $x=1/b$ yields
\[
 M=\frac{1/b}{1+1/b^2}<Q<\frac1b.
\]
Moreover, since $0<t/b<1$, the second estimate gives
\[
 J=\frac12\log\frac{1+t/b}{1-t/b}>\frac tb>0.
\]
Thus
\[
 N_1=Q+tJ>0.
\]
To treat $N_2$, we rearrange its terms and use $J>t/b$:
\[
 \begin{aligned}
 N_2
 &=(1-t^2)Q+2tJ-(1+t^2)M\\
 &=(Q-M)+2tJ-t^2(Q+M)\\
 &>(Q-M)+t^2\left(\frac2b-Q-M\right)\\
 &=(Q-M)+t^2\left[
       \left(\frac1b-Q\right)
       +\left(\frac1b-M\right)
     \right]>0.
 \end{aligned}
\]
The first strict inequality uses $t>0$ and $J>t/b$.
The final expression is strictly positive because
$M<Q<1/b$. This argument applies to every $0<t<b$.

Hence $N_1>0$ and $N_2>0$ in both cases, and therefore
$I_1>0$ and $I_2>0$.

We now justify the use of \eqref{eq:new-sinc-mixture} under the
integral sign. Fix $y,t>0$, and write
\[
 f_{p,t}(x)=(E^p\cos(t\,\cdot))(x),
 \qquad p=0,1,2.
\]
For any bounded function $f$ on $[0,\infty)$, the definition of $E$
gives
\[
 |(Ef)(x)|
 \leq \|f\|_\infty\int_0^x e^{-(x-v)}\dd v
 =(1-e^{-x})\|f\|_\infty
 \leq \|f\|_\infty.
\]
Thus $E$ is a contraction in the supremum norm, and consequently
$|f_{p,t}(x)|\leq1$ for all $x\geq0$ and $p=0,1,2$.

For $0<\varepsilon\leq1$, introduce the damped integrals
\[
 I_{p,\varepsilon}(b,t)
 =\int_0^{+\infty} e^{-\varepsilon x}
       \frac{\sin(bx)}x f_{p,t}(x)\dd x.
\]
Since $|\sin(bx)|\leq bx$, we have
\[
 \begin{aligned}
 &\int_0^\infty\int_0^\infty
   e^{-\varepsilon x}\frac{|\sin(bx)|}{x}
   |f_{p,t}(x)|w_y(b)\dd b\dd x\\
 &\qquad\leq
   \int_0^\infty e^{-\varepsilon x}\dd x
   \int_0^\infty b\,w_y(b)\dd b\\
 &\qquad=
   \frac1\varepsilon\int_0^\infty b\,w_y(b)\dd b
   <\infty.
 \end{aligned}
\]
Therefore Fubini's theorem applies for each fixed
$\varepsilon>0$, and \eqref{eq:new-sinc-mixture} yields
\[
 \int_0^{+\infty} e^{-\varepsilon x}K_y(x)f_{p,t}(x)\dd x
 =
 \int_0^{+\infty} I_{p,\varepsilon}(b,t)w_y(b)\dd b.
\]
The preceding absolute-integrability bound depends on
$\varepsilon$. To pass to the limit as $\varepsilon\downarrow0$,
we next obtain a bound for $I_{p,\varepsilon}(b,t)$ that is
independent of $\varepsilon$.

For this purpose, define
\[
 \begin{aligned}
 C_\varepsilon
 &=
 \int_0^{+\infty} e^{-\varepsilon x}
       \frac{\sin(bx)\cos(tx)}x\dd x\\
 &=
 \frac12\left(
   \arctan\frac{b+t}{\varepsilon}
   +\arctan\frac{b-t}{\varepsilon}
 \right),\\
 J_\varepsilon
 &=
 \int_0^\infty e^{-\varepsilon x}
       \frac{\sin(bx)\sin(tx)}x\dd x\\
 &=
 \frac14\log
 \frac{\varepsilon^2+(b+t)^2}
      {\varepsilon^2+(b-t)^2},\\
 L_\varepsilon
 &=
 \int_0^\infty e^{-(1+\varepsilon)x}
       \frac{\sin(bx)}x\dd x
 =
 \arctan\frac{b}{1+\varepsilon},\\
 M_\varepsilon
 &=
 \int_0^\infty e^{-(1+\varepsilon)x}
       \sin(bx)\dd x
 =
 \frac{b}{(1+\varepsilon)^2+b^2}.
 \end{aligned}
\]
These formulas follow from the product-to-sum identities and
the elementary Laplace integrals. Taking real parts in the
formula for $E^p e^{itx}$ displayed above gives
\[
 \begin{aligned}
 I_{0,\varepsilon}(b,t)
 &=C_\varepsilon,\\
 I_{1,\varepsilon}(b,t)
 &=
 \frac{C_\varepsilon-L_\varepsilon+tJ_\varepsilon}
      {1+t^2},\\
 I_{2,\varepsilon}(b,t)
 &=
 \frac{(1-t^2)(C_\varepsilon-L_\varepsilon)
       +2tJ_\varepsilon-(1+t^2)M_\varepsilon}
      {(1+t^2)^2}.
 \end{aligned}
\]

We have
\[
 |C_\varepsilon|\leq\frac{\pi}{2},
 \qquad
 0\leq L_\varepsilon\leq\frac{\pi}{2},
 \qquad
 0\leq M_\varepsilon
 \leq\frac{1}{2(1+\varepsilon)}
 \leq\frac12.
\]
For $b\ne t$, the inequality
\[
 1\leq
 \frac{\varepsilon^2+(b+t)^2}
      {\varepsilon^2+(b-t)^2}
 \leq
 \frac{(b+t)^2}{(b-t)^2}
\]
also gives
\[
 0\leq J_\varepsilon
 \leq \frac12\log\frac{b+t}{|b-t|}
 =J.
\]
Consequently,
\[
 \begin{aligned}
 |I_{0,\varepsilon}(b,t)|
 &\leq\frac{\pi}{2},\\
 |I_{1,\varepsilon}(b,t)|
 &\leq\frac{\pi+tJ}{1+t^2},\\
 |I_{2,\varepsilon}(b,t)|
 &\leq
 \frac{\pi|1-t^2|+2tJ+(1+t^2)/2}
      {(1+t^2)^2}.
 \end{aligned}
\]
In particular, there is a constant $C_t>0$, depending only on
$t$, such that
\[
 |I_{p,\varepsilon}(b,t)|
 \leq C_t(1+J),
 \qquad
 p=0,1,2,\quad b>0,\quad b\ne t,\quad
 0<\varepsilon\leq1.
\]

We verify that this majorant is integrable with respect to
$w_y(b)\dd b$. The elementary estimate
\[
 \frac12\log\frac{1+u}{1-u}
 =\int_0^u\frac{\dd s}{1-s^2}
 \leq 2u,
 \qquad 0\leq u\leq\frac12,
\]
implies
\[
 0\leq J\leq
 \begin{cases}
  2b/t, & 0<b\leq t/2,\\[2mm]
  2t/b, & b\geq2t.
 \end{cases}
\]
Thus $J$ is bounded on these two ranges, and the integrability
there follows from $\int_0^\infty w_y(b)\dd b<\infty$.
On the remaining interval $[t/2,2t]$, the function $w_y$ is
bounded by continuity, while
\[
 0\leq J
 \leq\frac12\log\frac{3t}{|b-t|}
 \qquad (t/2\leq b\leq2t,\ b\ne t).
\]
The singularity at $b=t$ is integrable, since
$\int_0^\delta|\log u|\dd u<\infty$ for every $\delta>0$.
It follows that
\[
 \int_0^\infty (1+J)w_y(b)\dd b<\infty,
\]
where the value of $J$ at the single point $b=t$ may be
assigned arbitrarily.

For each fixed $b\ne t$, the explicit formulas above show that
\[
 C_\varepsilon\rightarrow C,\qquad
 J_\varepsilon\rightarrow J,\qquad
 L_\varepsilon\rightarrow L,\qquad
 M_\varepsilon\rightarrow M
 \qquad(\varepsilon\downarrow0).
\]
Hence
\[
 I_{p,\varepsilon}(b,t)\rightarrow I_p(b,t),
\]
where $I_p(b,t)$ denotes the quantity $I_p$ computed above.
The measure $w_y(b)\dd b$ is absolutely continuous with respect
to Lebesgue measure, so $\{t\}$ has measure zero. No convergence
at $b=t$ is needed, and $I_p(b,t)$ may also be assigned an
arbitrary value there. The dominated convergence theorem
therefore gives
\[
 \lim_{\varepsilon\downarrow0}
 \int_0^\infty I_{p,\varepsilon}(b,t)w_y(b)\dd b
 =
 \int_0^\infty I_p(b,t)w_y(b)\dd b.
\]

We can also remove the damping from the original $x$-integral.
Indeed, $K_y(x)\to y^2$ as $x\downarrow0$, and, for $x\geq1$,
\[
 |K_y(x)|
 \leq\frac{2+y}{e^x-1}
 \leq\frac{2+y}{1-e^{-1}}e^{-x}.
\]
Thus $K_y\in L^1(0,+\infty)$. Since
\[
 |e^{-\varepsilon x}K_y(x)f_{p,t}(x)|
 \leq |K_y(x)|,
\]
another application of dominated convergence yields
\[
 \lim_{\varepsilon\downarrow0}
 \int_0^\infty e^{-\varepsilon x}K_y(x)f_{p,t}(x)\dd x
 =
 \int_0^\infty K_y(x)f_{p,t}(x)\dd x.
\]
Passing to the limit on both sides of the damped identity
therefore proves the required interchange:
\[
 \int_0^\infty K_y(x)(E^p\cos(t\,\cdot))(x)\dd x
 =
 \int_0^\infty I_p(b,t)w_y(b)\dd b,
 \qquad p=0,1,2.
\]
The integral on the right is absolutely convergent by the
majorant established above.

For $p=1,2$, we have proved that $I_p(b,t)>0$ for every
$b>0$ with $b\ne t$. Moreover, $w_y$ is continuous,
nonnegative, and not identically zero, so it is positive on
a set of positive Lebesgue measure. The last identity
therefore implies
\[
 \int_0^\infty K_y(x)(E^p\cos(t\,\cdot))(x)\dd x>0,
 \qquad p=1,2.
\]
For $p=0$, the formula
$I_0(b,t)=(\pi/2)\mathbf1_{\{t<b\}}$ gives
\[
 \begin{aligned}
 \int_0^\infty K_y(x)\cos(tx)\dd x
 &=
 \frac{\pi}{2}\int_t^\infty w_y(b)\dd b\\
 &=
 -\pi\int_t^\infty k_y'(b)\dd b\\
 &=\pi k_y(t)>0,
 \end{aligned}
\]
where we used $w_y=-2k_y'$ and
$\lim_{b\to\infty}k_y(b)=0$.$\hfill\Box$\\

\noindent{\bf Proof of Lemma~\ref{lem:new-summation}:}
We first establish the operator inequalities for the auxiliary
atoms $\gamma_t$. The series defining $\mathcal D_p\gamma_t$
converges locally uniformly for $\Rea s>0$ and extends
continuously to $\Rea s=1$. Indeed, for fixed $t$ and $s$ in a
compact subset the paired summand is $O(n^{-p-2})$, by the
resolvent formula for $\gamma_t$ above. All shifted arguments
$n+s$ stay in $\Rea z>1$.
On the boundary $s=1+iy$, $y>0$, reindexing gives
\[
 (\mathcal D_p\gamma_t)(1+iy)
 =\sum_{n=2}^\infty
 \left(\frac{\gamma_t(n)}{n^p}
             -\frac{\gamma_t(n+iy)}{(n+iy)^p}\right).
\]
Here $\gamma_t(1)=0$. The reindexing is made first in partial
sums; the final term tends to zero, also when $p=0$.
For $p\geq3$, Lemma~\ref{lem:new-atom-bound} makes the real part
of every summand strictly positive.

For $p=0,1,2$, we use the following Laplace identity:
\[
 \frac{\gamma_t(z)}{z^p}
 =\int_0^\infty e^{-(z-1)x}
       (E^p\cos(t\,\cdot))(x)\dd x,
 \qquad \Rea z>1,
\]
where $E^0$ denotes the identity operator. To verify it, first
write $$\cos(tx)=(e^{itx}+e^{-itx})/2.$$ Since $\Rea z>1$,
\[
 \begin{aligned}
 \int_0^\infty e^{-(z-1)x}\cos(tx)\dd x
 &=\frac12\left(
       \frac1{z-1-it}+\frac1{z-1+it}\right)\\
 &=\frac{z-1}{(z-1)^2+t^2}
 =\gamma_t(z).
 \end{aligned}
\]
Next, let $f$ be a bounded measurable function on $[0,+\infty)$
and put $\sigma=\Rea z>1$. The estimate
\[
 \begin{aligned}
 &\int_0^{+\infty}\int_0^x
   \left|e^{-(z-1)x}e^{-(x-v)}f(v)\right|
   \dd v\dd x\\
 &\qquad\leq
 \|f\|_\infty\int_0^{+\infty}
        e^{-(\sigma-1)x}(1-e^{-x})\dd x\\
 &\qquad=
 \frac{\|f\|_\infty}{\sigma(\sigma-1)}
 <\infty
 \end{aligned}
\]
justifies Fubini's theorem. Hence
\[
 \begin{aligned}
 \int_0^\infty e^{-(z-1)x}(Ef)(x)\dd x
 &=\int_0^\infty f(v)
       \left(\int_v^\infty
          e^{-(z-1)x}e^{-(x-v)}\dd x\right)\dd v\\
 &=\int_0^\infty f(v)e^v
       \left(\int_v^\infty e^{-zx}\dd x\right)\dd v\\
 &=\frac1z\int_0^\infty
       e^{-(z-1)v}f(v)\dd v.
 \end{aligned}
\]
Moreover,
\[
 |(Ef)(x)|
 \leq (1-e^{-x})\|f\|_\infty
 \leq \|f\|_\infty.
\]
Thus $E^j\cos(t\,\cdot)$ is bounded for every nonnegative
integer $j$. Applying the preceding identity successively
with $f=\cos(t\,\cdot)$ and $f=E\cos(t\,\cdot)$ gives
\[
 \begin{aligned}
 \int_0^\infty e^{-(z-1)x}
       (E\cos(t\,\cdot))(x)\dd x
 &=\frac{\gamma_t(z)}z,\\
 \int_0^\infty e^{-(z-1)x}
       (E^2\cos(t\,\cdot))(x)\dd x
 &=\frac{\gamma_t(z)}{z^2}.
 \end{aligned}
\]
Together with the cosine transform, these equalities prove
the claimed formula for $p=0,1,2$.

Fix $y,t>0$ and $p\in\{0,1,2\}$, and put
\[
 f_{p,t}(x)=(E^p\cos(t\,\cdot))(x).
\]
Since $E$ has a real kernel and is a contraction in the
supremum norm, $f_{p,t}$ is real-valued and
$|f_{p,t}(x)|\leq1$. Applying the Laplace identity at
$z=n$ and $z=n+iy$, we obtain, for every $n\geq2$,
\[
 \frac{\gamma_t(n)}{n^p}
 -\frac{\gamma_t(n+iy)}{(n+iy)^p}
 =
 \int_0^\infty e^{-(n-1)x}
       (1-e^{-iyx})f_{p,t}(x)\dd x.
\]
The preceding absolute-convergence estimate permits us to
interchange summation and integration in the boundary series.
Consequently,
\[
 \begin{aligned}
 (\mathcal D_p\gamma_t)(1+iy)
 &=\int_0^\infty
       \left(\sum_{n=2}^\infty e^{-(n-1)x}\right)
       (1-e^{-iyx})f_{p,t}(x)\dd x\\
 &=\int_0^\infty
       \frac{1-e^{-iyx}}{e^x-1}f_{p,t}(x)\dd x,
 \end{aligned}
\]
where we used the geometric-series identity
\[
 \sum_{n=2}^\infty e^{-(n-1)x}
 =\frac{e^{-x}}{1-e^{-x}}
 =\frac1{e^x-1},
 \qquad x>0.
\]
Next, since $(1+iy)^{-1}=(1-iy)/(1+y^2)$ and
$1-e^{-iyx}=1-\cos(yx)+i\sin(yx)$, we have
\[
 \begin{aligned}
 \Rea\left(\frac{1-e^{-iyx}}{1+iy}\right)
 &=\frac{\Rea\bigl((1-iy)
       (1-\cos(yx)+i\sin(yx))\bigr)}{1+y^2}\\
 &=\frac{1-\cos(yx)+y\sin(yx)}{1+y^2}.
 \end{aligned}
\]
The absolute convergence of the integral and the fact that
$f_{p,t}$ is real-valued therefore give
\[
 \begin{aligned}
 \Rea\frac{(\mathcal D_p\gamma_t)(1+iy)}{1+iy}
 &=\int_0^\infty
       \Rea\left(\frac{1-e^{-iyx}}{1+iy}\right)
       \frac{f_{p,t}(x)}{e^x-1}\dd x\\
 &=\frac1{1+y^2}\int_0^\infty
       \frac{1-\cos(yx)+y\sin(yx)}{e^x-1}
       f_{p,t}(x)\dd x\\
 &=\frac1{1+y^2}\int_0^\infty
       K_y(x)(E^p\cos(t\,\cdot))(x)\dd x>0.
 \end{aligned}
\]
The last equality is the definition of $K_y$, and the strict
inequality follows from Lemma~\ref{lem:new-small-powers},
since $y,t>0$, $p=0,1,2$, and $1+y^2>0$.

We check the boundedness needed to pass from the boundary to the
interior. For $\Rea s\geq1$ and $n\geq1$,
\[
 |\gamma_t(n+s)|\leq\frac1{n+\Rea s-1}\leq\frac1n,
 \qquad |n+s|\geq n+1.
\]
Also $0<\gamma_t(n)\leq1/(n-1)$ for $n\geq2$.
Thus $\mathcal D_p\gamma_t$ is bounded there for every $p\geq1$.
For $p=0$, split the paired series at $n=2|s|+2$.
The initial terms have total absolute value
$O(1+\log(1+|s|))$ by these bounds, and the resolvent identity gives
\[
 |\gamma_t(n)-\gamma_t(n+s)|
 \leq\frac{|s|}{(n-1)(n+\Rea s-1)}\qquad(n\geq2).
\]
The remaining tail is $O(1)$. Hence
$(\mathcal D_0\gamma_t)(s)/s$ is bounded on $\Rea s\geq1$.
Define the bounded harmonic function
\[
 U_{p,t}(s)=
 \begin{cases}
 \Rea((\mathcal D_p\gamma_t)(s)/s),&p=0,1,2,\\
 \Rea(\mathcal D_p\gamma_t)(s),&p\geq3.
 \end{cases}
\]
It is continuous up to the boundary. The bounded harmonic
Poisson formula \cite[Theorem 7.5]{ABR} gives, for $\sigma>1$,
\[
 U_{p,t}(\sigma+iy)=\frac1\pi\int_{\mathbb R}
 \frac{\sigma-1}{(\sigma-1)^2+(y-v)^2}
 U_{p,t}(1+iv)\dd v>0.
\]
The last inequality is strict because the Poisson kernel is
positive and the boundary values are positive for every $v\ne0$.

Now let $g$ satisfy the hypotheses of the lemma, and put
$B(t)=\Rea g(1+it)$. The identity theorem applied to
$\overline{g(\bar s)}$ gives conjugation symmetry. Hence $B$ is even.
Continuity from $\Rea s>1$ gives $B\geq0$.
The representation theorem for positive harmonic functions
\cite[Theorem 7.26]{ABR}, applied in the shifted half-plane, gives
\[
 \Rea g(1+x+iy)
 =c x+\frac1\pi\int_{\mathbb R}
        \frac{x}{x^2+(y-t)^2}\,d\mu(t),\qquad x>0,
\]
where $c\geq0$ and $\mu$ is a positive Borel measure satisfying
$\int_{\mathbb R}(1+t^2)^{-1}d\mu(t)<\infty$.

We now show that the representing measure is
$d\mu(t)=B(t)\dd t$. For $x>0$, set
\[
 u_x(y)=\Rea g(1+x+iy),
 \qquad
 P_x(r)=\frac1\pi\frac{x}{x^2+r^2}.
\]
The preceding representation becomes
\[
 u_x(y)=cx+\int_{\mathbb R}P_x(y-t)\dd\mu(t).
\]
The weighted integrability of $\mu$ implies local finiteness,
since, for every $A>0$,
\[
 \mu([-A,A])
 \leq (1+A^2)
       \int_{\mathbb R}\frac{\dd\mu(t)}{1+t^2}
 <\infty.
\]

Fix an arbitrary real-valued continuous function
$\varphi$ with compact support in $\mathbb R$.
Since $g$ is holomorphic on $\Rea s>0$, it is continuous
across the line $\Rea s=1$, and
$u_x(y)\to B(y)$ uniformly on every compact interval as
$x\downarrow0$. Consequently,
\[
 \lim_{x\downarrow0}
 \int_{\mathbb R}\varphi(y)u_x(y)\dd y
 =
 \int_{\mathbb R}\varphi(y)B(y)\dd y.
\]

We compute the same limit from the Poisson representation.
Choose $R\geq1$ such that
$\operatorname{supp}\varphi\subseteq[-R,R]$, and write
\[
 (P_x*\varphi)(t)
 =\int_{\mathbb R}P_x(t-y)\varphi(y)\dd y,
 \qquad
 \|\varphi\|_1=\int_{\mathbb R}|\varphi(y)|\dd y.
\]
For $0<x\leq1$, the identity
$\int_{\mathbb R}P_x(r)\dd r=1$ gives
\[
 (P_x*|\varphi|)(t)\leq\|\varphi\|_\infty
 \qquad(t\in\mathbb R).
\]
If $|t|>2R$, then $|t-y|\geq |t|/2$ on the support of
$\varphi$, and hence
\[
 \begin{aligned}
 (P_x*|\varphi|)(t)
 &\leq \frac{4x}{\pi t^2}\|\varphi\|_1\\
 &\leq \frac{8\|\varphi\|_1}{\pi(1+t^2)}.
 \end{aligned}
\]
Combining the estimates on $|t|\leq2R$ and $|t|>2R$, we
obtain a constant $C_\varphi$, independent of $x$, such that
\[
 |(P_x*\varphi)(t)|
 \leq (P_x*|\varphi|)(t)
 \leq \frac{C_\varphi}{1+t^2},
 \qquad t\in\mathbb R,\quad 0<x\leq1.
\]
In particular, Tonelli's theorem and the weighted
integrability of $\mu$ yield
\[
 \begin{aligned}
 &\int_{\mathbb R}\int_{\mathbb R}
       |\varphi(y)|P_x(y-t)\dd\mu(t)\dd y\\
 &\qquad=
 \int_{\mathbb R}(P_x*|\varphi|)(t)\dd\mu(t)
 <\infty.
 \end{aligned}
\]
Thus Fubini's theorem, together with the evenness of $P_x$,
gives
\[
 \int_{\mathbb R}\varphi(y)u_x(y)\dd y
 =
 cx\int_{\mathbb R}\varphi(y)\dd y
 +\int_{\mathbb R}(P_x*\varphi)(t)\dd\mu(t).
\]

For each fixed $t\in\mathbb R$, the change of variables
$y=t-xr$ gives
\[
 (P_x*\varphi)(t)
 =
 \frac1\pi\int_{\mathbb R}
       \frac{\varphi(t-xr)}{1+r^2}\dd r
 \rightarrow \varphi(t)
 \qquad(x\downarrow0).
\]
This follows from the dominated convergence theorem,
because $\varphi$ is continuous and bounded, while
$(1+r^2)^{-1}$ is integrable on $\mathbb R$.
The uniform weighted bound established above now permits
another application of dominated convergence, this time
with respect to $\mu$, and gives
\[
 \lim_{x\downarrow0}
 \int_{\mathbb R}(P_x*\varphi)(t)\dd\mu(t)
 =
 \int_{\mathbb R}\varphi(t)\dd\mu(t).
\]
Since $cx\int_{\mathbb R}\varphi(y)\dd y\to0$, it follows that
\[
 \lim_{x\downarrow0}
 \int_{\mathbb R}\varphi(y)u_x(y)\dd y
 =
 \int_{\mathbb R}\varphi(t)\dd\mu(t).
\]
Comparison with the limit obtained from boundary continuity
therefore shows that
\[
 \int_{\mathbb R}\varphi(t)\dd\mu(t)
 =
 \int_{\mathbb R}\varphi(t)B(t)\dd t.
\]
This holds for every real-valued $\varphi\in C_c(\mathbb R)$.
Both $\mu$ and $B(t)\dd t$ are locally finite Borel measures
on $\mathbb R$, and such measures are uniquely determined
by their integrals against compactly supported continuous
functions. Hence $d\mu(t)=B(t)\dd t$; in particular, $\mu$
is absolutely continuous with respect to Lebesgue measure
and has no singular component.

Moreover, $\Rea g(1+x)\geq c x$, whereas the logarithmic growth
hypothesis gives $g(1+x)=o(x)$, so $c=0$.
The Taylor coefficients of $g$ at $1$ are real. Since $g(1)=0$,
we have $B(t)=O(t^2)$ at zero; at infinity the assumed growth
gives $B(t)=O(\log^M(2+t))$.
Consequently the following integral is absolutely and locally
uniformly convergent on $\Rea s>1$:
\begin{equation}\label{eq:new-herglotz}
 g(s)=\frac2\pi\int_0^{+\infty}\gamma_t(s)B(t)\dd t.
\end{equation}
To verify equality, pair the positive and negative $t$ terms
in the harmonic representation. The right-hand side is
holomorphic and has the same real part as $g$, so their
difference is an imaginary constant. Both are real for real
$s>1$, and that constant is therefore zero.
Also $B$ is not identically zero, since otherwise
\eqref{eq:new-herglotz} would contradict $\Rea g>0$.

It remains to justify applying $\mathcal D_p$ to this positive
integral. Let $K$ be a compact subset of $\Rea s>1$.
For $n\geq2$, the positive-real Schwarz estimate on the disc
$|z-n|<n-1$ gives
\[
 |\gamma_t(z)|\leq3\gamma_t(n)
 \qquad\left(|z-n|\leq\frac{n-1}{2}\right),
\]
uniformly in $t>0$. Indeed, apply Schwarz's lemma to
$(\gamma_t(z)-\gamma_t(n))/(\gamma_t(z)+\gamma_t(n))$,
which maps this disc into the unit disc and vanishes at $n$.
Cauchy's estimate on a smaller disc, followed by integration
along the segment from $n$ to $n+s$, now gives, for sufficiently
large $n$ uniformly in $s\in K$ and $t>0$,
\[
 \left|\frac{\gamma_t(n)}{n^p}
        -\frac{\gamma_t(n+s)}{(n+s)^p}\right|
 \leq C_{K,p}\frac{\gamma_t(n)}{n^{p+1}}.
\]
After integration in \eqref{eq:new-herglotz}, the majorant is
$C_{K,p}g(n)/n^{p+1}$. Its series converges: for $p\geq1$ use
$g(n)=O(\log^M(2+n))$, and for $p=0$ use the additional bound
$g(n)=O(\log^M(2+n)/n)$.
For the finitely many remaining terms, the positive-real Schwarz
estimate based at $z=2$ bounds the atom on each compact subset
of $\Rea z>1$ by a constant times $\gamma_t(2)$.
This is integrable against $B(t)\dd t$ by
\eqref{eq:new-herglotz} at $s=2$.
In particular, for $n=1$ the unshifted term is zero and
$|\gamma_t(1+s)|\leq C_K\gamma_t(2)$ handles the shifted term.

These bounds prove absolute, locally uniform convergence of the
paired series and absolute summability under the atom integral.
Fubini's theorem therefore gives
\[
 (\mathcal D_p g)(s)=\frac2\pi\int_0^\infty
       (\mathcal D_p\gamma_t)(s)B(t)\dd t.
\]
For each fixed $s$ in the half-plane, the appropriate normalized
real part of the atom output is strictly positive for every
$t>0$. Since the continuous density $B\geq0$ is not identically
zero, integration preserves strict positivity.
This proves Lemma~\ref{lem:new-summation} and completes the proof
of Theorem~\ref{thm:complex-injectivity}.
$\hfill\Box$\\

\section*{Acknowledgements }
This project is funded by the National Natural Science Foundation of China (Grant No. 12571009) and the
Natural Science Foundation of Hunan Province,
China (Grant No. 2026JJ40003). AI tools (GPT 6 Pro) is used for language polishing, proofreading and developing the proof of Theorem \ref{thm:complex-injectivity}.

\end{document}